\documentclass[10pt]{article}%
\usepackage{amsmath}
\usepackage{amsfonts}
\usepackage{mathrsfs}
\usepackage{amssymb,color}
\usepackage[linkcolor=black,anchorcolor=black,citecolor=black]{hyperref}
\usepackage{graphicx}
\numberwithin{equation}{section}
\usepackage[body={15.5cm,21cm}, top=3cm]{geometry}%
\providecommand{\U}[1]{\protect\rule{.1in}{.1in}}
\providecommand{\U}[1]{\protect \rule{.1in}{.1in}}
\newtheorem{theorem}{Theorem}[section]

\newtheorem{corollary}[theorem]{Corollary}

\newtheorem{definition}[theorem]{Definition}

\newtheorem{lemma}[theorem]{Lemma}

\newtheorem{proposition}[theorem]{Proposition}
\newtheorem{remark}[theorem]{Remark}
\newtheorem{assumption}[theorem]{Assumption}

\newenvironment{proof}[1][Proof]{\noindent \textbf{#1.} }{\  \rule{0.5em}{0.5em}}

\begin{document}
	\title{Reflected Stochastic Differential Equations Driven by $G$-Brownian Motion with Nonlinear Constraints}
	\author{ Hanwu Li\thanks{Research Center for Mathematics and Interdisciplinary Sciences, Shandong University, Qingdao 266237, Shandong, China. lihanwu@sdu.edu.cn.}
	\thanks{Frontiers Science Center for Nonlinear Expectations (Ministry of Education), Shandong University, Qingdao 266237, Shandong, China.}
    \thanks{Shandong Province Key Laboratory of Financial Risk, Shandong University, Qingdao 266237, Shandong, China.}}
	\date{}
	\maketitle
	\begin{abstract}
	In this paper, we study the reflected stochastic differential equations driven by $G$-Brownian motion (reflected $G$-SDEs) with two nonlinear constraints. With the help of the  Skorokhod problem with nonlinear constraints, we first study the doubly  reflected $G$-Brownian motion, which is constructed pathwise and lies in the same $G$-expectation space as the $G$-Brownian motion. For the reflected $G$-SDE, the uniqueness is derived from some a priori estimate and the existence is obtained by a Picard iteration method. The comparison theorem of the solution and the individual constraining processes are provided.
	\end{abstract}
	
	
	\textbf{MSC-classification}: 60G65, 60H10

\section{Introduction}
	 	In the 1960s, Skorokhod \cite{Skorokhod1} considered the problem of constructing solutions to  stochastic differential equations (SDEs) with reflecting boundaries.  Since then, a great deal of attention has been made to investigate the reflected solutions to stochastic differential equations. To name a few, Chaleyat-Maurel and El Karoui \cite{CE} studied  reflected SDEs on a half-line and used the Skorokhod map in a study of local times.  A multidimensional version of the Skorokhod problem was introduced by Tanaka \cite{Tanaka}, where the solutions are required to take values in a convex domain. Then, Lions and Sznitman \cite{LS} extended the results to the non-convex case. Slomi\'{n}ski \cite{Sl1,Sl2} considered the multidimensional reflected SDEs driven by semimartingales and proposed an Euler-Peano scheme. Jarni and Ouknine \cite{JO} investigated the reflected SDEs driven by optional semimartingales on  an unusual probability space.
		
		Due to the importance in applications  including queuing theory, control theory and finance (see, e.g. \cite{BN,EK,HH,MM,SW,W}), the Skorokhod problem has attracted considerable attention. One of the branches of the Skorokhod problem is the study of the case when there are two reflecting boundaries. Kruk et al. \cite{KLRS} presented an explicit formula for the Skorokhod map on a constant interval $[0,a]$. Then, Burdzy et al. \cite{BKR} investigated the Skorokhod problem in a time-dependent interval and provided an explicit representation for the so-called extended Skorokhod mapping. By the method based on the approach in \cite{KLRS}, Slaby \cite{S1} provided another representation formula for the Skorokhod problem with two time-dependent boundaries.  An alternative explicit formula for the extended Skorokhod mapping was given in \cite{S2}. 	 Slomi\'{n}ski and Wojciechowski \cite{SW1,SW2} introduced the SDEs with time-dependent reflecting barriers. Recently, Li \cite{Li} considered the Skorokhod problem with two nonlinear constraints. 
		
		It is worth mentioning that all the above literature considers the Skorokhod problem either in the deterministic setting or in the classical probability space. Therefore, they can not deal with the problems facing Knightian uncertainty, especially volatility uncertainty. One of the most powerful tools to study problems under volatility uncertainty is the $G$-expectation established by Peng \cite{P19}. Roughly speaking, the $G$-expectation is an upper expectation taking over a non-dominated family of probability measures. A new type of Brownian motion with stationary and independent increments, called $G$-Brownian motion, was constructed and the associated It\^{o}'s calculus was established. Building upon this framework, Gao \cite{G} investigated the stochastic differential equations driven by $G$-Brownian motion ($G$-SDEs for short) under a standard Lipschitz condition. Lin \cite{Lin} first tackled the reflecting problem where the solution of $G$-SDE is required to be above some given process. The multidimensional case is solved in \cite{Lin2} by the penalization method. Then, Lin and Soumana Hima \cite{LSH} generalized the above results to the case when the reflecting boundary is not convex and Soumana Hima and Dakaou \cite{SHD} established a large deviation principle for the solution. When the increasing process $K$ contributes to the coefficients, we may refer to the paper by Luo \cite{Luo}. Li and Ning \cite{LN23} considered the mean reflected $G$-SDE, where the constraint is given in terms of the expectation of the solution.

		The objective of this paper is to study the solvability of reflected $G$-SDEs with two nonlinear constraints. More precisely, the solution with initial value $x$, coefficients $f,h,g$ and constraints $l,r$ is a pair of processes $(X,A)$ satisfying 
		\begin{equation}\label{eq1.1}
		\begin{cases}
		X_t=x+\int_0^t f(s,X_s)ds+\int_0^t h(x,X_s)d\langle B\rangle_s+\int_0^t g(s,X_s)dB_s+A_t, \ t\in[0,T];\\
                 l(t,X_t)\leq 0\leq r(t,X_t), t\in[0,T], \textrm{ q.s.};\\
                 A_0=0, A=A^r-A^l \textrm{ and }
\int_0^T r(s,X_s)dA^r_s=\int_0^T l(s,X_s)dA^l_s=0, \textrm{ q.s.},\\
                \end{cases}
		\end{equation}
where $A^r$, $A^l$ are two increasing processes, called the constraining processes, such that $A^r$ pushes the solution upward and $A^l$ pulls the solution downward. They should behave in a minimal way satisfying the Skorokhod condition. 

In order to establish the well-posedness of the previous reflected $G$-SDE, we first consider the case of constant coefficients. The solution is constructed pathwise using the Skorokhod problem with two nonlinear constraints. However, an important feature in our setting is that the constraint functions $l,r$ are no longer bi-Lipschitz, which makes it unable to apply the results in \cite{Li} directly. The first contribution of the paper is that we investigate the Skorokhod problem with weaker conditions on the nonlinear constraints, establishing the existence and uniqueness result and the continuity property with respect to the input process and the inverse of the constraint functions. On the other hand, recall that the solution to a non-reflected $G$-SDE belong to the space $S^p_G(0,T)$ with continuous sample path. While for the reflected case studied in \cite{Lin,LSH,Luo}, the resulting process stays in the space $M_G^p(0,T)$, which consists of c\`{a}dl\`{a}g processes. The second contribution is that we establish a technical Lemma \ref{SGP}, which ensures the continuity of the solutions to the reflected $G$-SDEs. The existence of reflected $G$-SDEs with general Lipschitz coefficients can be proved by a fixed-point iteration method. Besides, based on the a priori estimates, we obtain the uniqueness result. Applying the extended $G$-It\^{o}'s formula proposed in \cite{Lin} and the characterization of the spaces $L_G^p(\Omega_T)$ and $M_G^p(0,T)$, we could also obtain the comparison theorem. Actually, the pathwise construction also allows us to derive some monotonicity results of the individual constraining processes.

The paper is organized as follows. In Section 2, we first introduce some notations and results in the $G$-framework and the Skorokhod problem with nonlinear constraints. Then, we study the reflected $G$-Brownian motion and some monotonicity properties with respect to the reflecting constraints in Section 3. Section 4 is devoted to the study of $G$-SDEs with nonlinear constraints, including the existence and uniqueness result as well as the comparison theorem.

\section{Preliminaries}

In this section, we review some basic notions and results of $G$-expectation, $G$-stochastic calculus and the Skorokhod problem with two reflecting boundaries. 

\subsection{$G$-expectation and $G$-It\^{o}'s calculus}

	Let $\Omega_T=C_{0}([0,T];\mathbb{R})$, the space of
real-valued continuous functions with $\omega_0=0$, be endowed
with the supremum norm and 
let  $B$ be the canonical
process. Set
\[
L_{ip} (\Omega_T):=\{ \varphi(B_{t_{1}},...,B_{t_{n}}):  \ n\in\mathbb {N}, \ t_{1}
,\cdots, t_{n}\in\lbrack0,T], \ \varphi\in C_{l,Lip}(\mathbb{R}^{ n})\},
\]
where $C_{l,Lip}(\mathbb{R}^{ n})$ denotes the set of local Lipschitz functions on $\mathbb{R}^{n}$.

Let $G:\mathbb{R}\rightarrow\mathbb{R}$ be a sublinear, continuous and monotone function  defined by
\begin{displaymath}
G(a):=\frac{1}{2}(\bar{\sigma}^2a^+-\underline{\sigma}^2a^-),
\end{displaymath}
where $0\leq \underline{\sigma}^2<\bar{\sigma}^2$. 

\begin{definition}
A scalar valued random variable $\xi\in L_{ip}(\Omega_T)$ is called $G$-normal distributed, i.e., $\xi\sim \mathcal{N}(0,[\underline{\sigma}^2,\bar{\sigma}^2])$, if for each $\varphi\in C_{l,Lip}(\mathbb{R})$, $u(t,x):=\hat{\mathbb{E}}[\varphi(x+\sqrt{t}\xi)]$ is a viscosity solution to the following PDE :
\begin{displaymath}
\begin{cases}
\partial_t u-G(\partial_x^2 u)=0,  &(t,x)\in\mathbb{R}^+\times\mathbb{R};\\
u(0,x)=\varphi(x), &x\in\mathbb{R}.
\end{cases}
\end{displaymath}
\end{definition}

\begin{definition}
We call a sublinear expectation $\hat{\mathbb{E}}:L_{ip}(\Omega_T)\rightarrow \mathbb{R}$ a $G$-expectation if the canonical process $G$ is a $G$-Brownian motion under $\hat{\mathbb{E}}[\cdot]$, i.e., for each $0\leq s\leq t\leq T$, the increment $B_t-B_s\sim \mathcal{N}(0,[(t-s)\underline{\sigma}^2,(t-s)\bar{\sigma}^2])$ and for all $n>0$, $0\leq t_1\leq \cdots\leq t_n\leq T$ and $\varphi\in C_{l,Lip}(\mathbb{R}^n)$,
\begin{align*}
\hat{\mathbb{E}}[\varphi(B_{t_1},\cdots,B_{t_n-1},B_{t_n}-B_{t_n-1})]=\hat{\mathbb{E}}[\psi(B_{t_1},\cdots,B_{t_n-1})],
\end{align*}
where $\psi(x_1,\cdots,x_{n-1})=\hat{\mathbb{E}}[\varphi(x_{t_1},\cdots,x_{t_n-1},B_{t_n}-B_{t_n-1})]$.  The triple $(\Omega,L_{ip}(\Omega_T),\hat{\mathbb{E}})$ is called the $G$-expectation space.
\end{definition}


Define $\Vert\xi\Vert_{L_{G}^{p}}:=(\hat{\mathbb{E}}[|\xi|^{p}])^{1/p}$ for $\xi\in L_{ip}(\Omega_T)$ and $p\geq1$.   The completion of $L_{ip} (\Omega_T)$ under this norm  is denote by $L_{G}^{p}(\Omega_T)$. We may check that $\hat{\mathbb{E}}[\cdot]$ is a continuous mapping on $L_{ip}(\Omega_T)$ w.r.t the norm $\|\cdot\|_{L_G^1}$. Hence, the  $G$-expectation $\mathbb{\hat{E}}[\cdot]$ can be extended continuously to the completion $L_{G}^{1}(\Omega_T)$. Denis, Hu and Peng \cite{DHP11} prove that the $G$-expectation has the following representation.
\begin{theorem}[\cite{DHP11}]
	\label{the1.1}  There exists a weakly compact set
	$\mathcal{P}$ of probability
	measures on $(\Omega_T,\mathcal{B}(\Omega_T))$, such that
	\[
	\hat{\mathbb{E}}[\xi]=\sup_{P\in\mathcal{P}}E_{P}[\xi] \text{ for all } \xi\in  {L}_{G}^{1}{(\Omega_T)}.
	\]
	$\mathcal{P}$ is called a set that represents $\hat{\mathbb{E}}$.
\end{theorem}

Let $\mathcal{P}$ be a weakly compact set that represents $\hat{\mathbb{E}}$.
For this $\mathcal{P}$, we define the capacity%
\[
c(A):=\sup_{P\in\mathcal{P}}P(A),\ A\in\mathcal{B}(\Omega_T).
\]
A set $A\in\mathcal{B}(\Omega_T)$ is called polar if $c(A)=0$.  A
property holds $``quasi$-$surely"$ (q.s.) if it holds outside a
polar set. In the following, we do not distinguish two random variables $X$ and $Y$ if $X=Y$, q.s.

\begin{definition}
	\label{def2.6} Let $M_{G}^{0}(0,T)$ be the collection of processes in the
	following form: for a given partition $\{t_{0},\cdot\cdot\cdot,t_{N}\}=\pi
	_{T}$ of $[0,T]$,
	\[
	\eta_{t}(\omega)=\sum_{j=0}^{N-1}\xi_{j}(\omega)\mathbf{1}_{[t_{j},t_{j+1})}(t),
	\]
	where $\xi_{i}\in L_{ip}(\Omega_{t_{i}})$, $i=0,1,2,\cdot\cdot\cdot,N-1$. For each
	$p\geq1$ and $\eta\in M_G^0(0,T)$, let $\|\eta\|_{H_G^p}:=\{\hat{\mathbb{E}}[(\int_0^T|\eta_s|^2ds)^{p/2}]\}^{1/p}$, $\Vert\eta\Vert_{M_{G}^{p}}:=(\mathbb{\hat{E}}[\int_{0}^{T}|\eta_{s}|^{p}ds])^{1/p}$ and denote by $H_G^p(0,T)$,  $M_{G}^{p}(0,T)$ the completion
	of $M_{G}^{0}(0,T)$ under the norm $\|\cdot\|_{H_G^p}$, $\|\cdot\|_{M_G^p}$, respectively.
\end{definition}

We give the characterization of the space $L_G^p(\Omega_T)$ and $M_G^p(0,T)$ for each $p\geq 1$, respectively. To this end, set $\mathcal{F}_t=\mathcal{B}(\Omega_t)$ for $t\in[0,T]$ and the distance
\begin{displaymath}
\rho((t,\omega),(t',\omega'))=|t-t'|+\max_{s\in[0,T]}|\omega_s-\omega'_s|, \textrm{ for } (t,\omega), (t',\omega')\in [0,T]\times \Omega_T.
\end{displaymath}
We define
\begin{itemize}
\item $L^0(\Omega_T)$: the space of all $\mathcal{B}(\Omega_T)$-measurable functions;
\item $\mathbb{L}^p(\Omega_T):=\{X\in L^0(\Omega_T):\sup_{P\in\mathcal{P}}E_P[|X|^p]<\infty\}$;
\item $\mathbb{M}^p(0,T):=\{\eta: \textrm{progressively measurable on } [0,T]\times \Omega_T \textrm{ and } \hat{\mathbb{E}}\left[\int_0^T |\eta_t|^p dt\right]<\infty\}$;
\item $\hat{c}(A):=\frac{1}{T}\hat{\mathbb{E}}\left[\int_0^T I_A(t,\omega) dt\right], \textrm{ for each progressively measurable set } A\subset [0,T]\times \Omega_T$.
\end{itemize}


\begin{definition}
(i) A mapping $X:\Omega_T\rightarrow \mathbb{R}$ is said to be quasi-continuous if, for each $\varepsilon>0$, there exists an open set $O$ in $\Omega_T$ such that ${c}(G)<\varepsilon$ and $X|_{O^c}$ is continuous.

\noindent (ii) We say that $X:\Omega_T\rightarrow \mathbb{R}$ has a quasi-continuous version if there exists a quasi-continuous function $Y:\Omega_T\rightarrow \mathbb{R}$ such that $X=Y$, q.s.
\end{definition}

\begin{theorem}[\cite{DHP11}]\label{LGp}
For each $p\geq 1$, we have
\begin{displaymath}
L_G^p(\Omega_T):=\left\{X\in \mathbb{L}^p(\Omega_T): X \textrm{ has a quasi-continuous version and } \lim_{N\rightarrow\infty}\hat{\mathbb{E}}\left[ |X|^pI_{\{|X|\geq N\}} dt\right]=0\right\}.
\end{displaymath}
\end{theorem}

\begin{definition}
(i) A progressively measurable process $\eta:[0,T]\times \Omega_T\rightarrow \mathbb{R}$ is called quasi-continuous (q.c. for short), if for each $\varepsilon>0$, there exists a progressively measurable open set $G$ in $[0,T]\times \Omega_T$ such that $\hat{c}(G)<\varepsilon$ and $\eta|_{G^c}$ is continuous.

\noindent (ii) We say that a progressively measurable process $\eta:[0,T]\times \Omega_T\rightarrow \mathbb{R}$ has a quasi-continuous version if there exists a quasi-continuous process $\eta'$ such that $\hat{c}(\{\eta\neq \eta'\})=0$.
\end{definition}

\begin{theorem}[\cite{HWZ}]\label{thm4.7}
For each $p\geq 1$, we have
\begin{displaymath}
M_G^p(0,T):=\left\{\eta\in \mathbb{M}^p(0,T): \eta \textrm{ has a quasi-continuous version and } \lim_{N\rightarrow\infty}\hat{\mathbb{E}}\left[\int_0^T |\eta_t|^pI_{\{|\eta_t|\geq N\}} dt\right]=0\right\}.
\end{displaymath}
\end{theorem}

 We denote by $\langle B\rangle$ the quadratic variation process of the $G$-Brownian motion $B$. For two processes $ \xi\in M_{G}^{1}(0,T)$ and $ \eta\in M_{G}^{2}(0,T)$,
the $G$-It\^{o} integrals $(\int^{t}_0\xi_sd\langle
B\rangle_s)_{0\leq t\leq T}$ and $(\int^{t}_0\eta_sdB_s)_{0\leq t\leq T}$ are well defined, see  Li and Peng \cite{lp} and Peng \cite{P19}. The following proposition can be regarded as the Burkholder--Davis--Gundy inequality under $G$-expectation framework
\begin{proposition}[\cite{P19}]\label{the1.3}
	If $\eta\in H_G^{\alpha}(0,T)$ with $\alpha\geq 1$ and $p\in(0,\alpha]$, then we have
	\begin{displaymath}
	\underline{\sigma}^p c_p\hat{\mathbb{E}}\left[\left(\int_0^T |\eta_s|^2ds\right)^{p/2}\right]\leq
	\hat{\mathbb{E}}\left[\sup_{u\in[t,T]}\left|\int_t^u\eta_s dB_s\right|^p\right]\leq
	\bar{\sigma}^p C_p\hat{\mathbb{E}}\left[\left(\int_0^T |\eta_s|^2ds\right)^{p/2}\right],
	\end{displaymath}
	where $0<c_p<C_p<\infty$ are constants depending on $p, T$.
\end{proposition}

Let $S_G^0(0,T)=\{h(t,B_{t_1\wedge t}, \ldots,B_{t_n\wedge t}):t_1,\ldots,t_n\in[0,T],h\in C_{b,Lip}(\mathbb{R}^{n+1})\}$. For $p\geq 1$ and $\eta\in S_G^0(0,T)$, set $\|\eta\|_{S_G^p}=\{\hat{\mathbb{E}}[\sup_{t\in[0,T]}|\eta_t|^p]\}^{1/p}$. Denote by $S_G^p(0,T)$ the completion of $S_G^0(0,T)$ under the norm $\|\cdot\|_{S_G^p}$.  We have the following continuity property for any $Y\in S_G^p(0,T)$ with $p\geq 1$.

\begin{proposition}[\cite{LPS}]\label{the3.7}
For $Y\in S_G^p(0,T)$ with $p\geq 1$, we have, by setting $Y_s:=Y_T$ for $s>T$,
\begin{displaymath}
\limsup_{\varepsilon\rightarrow0}\hat{\mathbb{E}}\left[\sup_{t\in[0,T]}\sup_{s\in[t,t+\varepsilon]}|Y_t-Y_s|^p\right]=0.
\end{displaymath}
\end{proposition}


\subsection{Stochastic calculus with respect to an increasing process}

In this subsection, we introduce some notions and definitions of stochastic calculus with respect to an increasing process in \cite{Lin}.

\begin{definition}
	We denote by $M_c(0,T)$ the collection of all q.s. continuous processes $X$ whose paths $X_\cdot(\omega):[0,T]\rightarrow \mathbb{R}$ are continuous outside a polar set $A$ and by $M_I(0,T)$ the collection of all q.s. increasing processes $K\in M_c(0,T)$ whose paths $K_\cdot(\omega):[0,T]\rightarrow \mathbb{R}$ are increasing outside a polar set $A$.
	\end{definition}
	
\begin{definition}\label{deflin}
We define, for a fixed $X\in M_c(0,T)$, the stochastic integral with respect to a given $K\in M_I(0,T)$ by
\begin{displaymath}
(\int_0^T X_t dK_t)(\omega)=
\begin{cases}
\int_0^T X_t(\omega)dK_t(\omega), &\omega \in A^c;\\
0, &\omega\in A,
\end{cases}
\end{displaymath}
where $A$ is a polar set and on the complementary of which, $X_\cdot(\omega)$ is continuous and $K_\cdot(\omega)$ is continuous and increasing.
\end{definition}

As shown in \cite{Lin}, the integral $\int_0^T X_t dK_t\in L^0(\Omega_T)$ is well defined for any fixed $X\in M_c(0,T)$ and $K\in M_I(0,T)$. Unfortunately, this does not ensure the quasi-continuity of $\int_0^T X_t dK_t$.  Recalling the characterization  in Theorem \ref{LGp}, we cannot conclude that $\int_0^T X_t dK_t$ belong to $L_G^1(\Omega_T)$ when merely assuming $X\in M_c(0,T)$ and $K\in M_I(0,T)$.  Actually, Proposition 3.12 in \cite{Lin} gives an answer to the natural question that under what condition, $\int_0^T X_t dK_t$ belongs to $L_G^1(\Omega_T)$. However, the conditions are somewhat restrictive since the integrand $X$ is required to be a $G$-It\^{o} process. We provide another answer to this question,  which can be found in the Appendix. In the following, we provide an extension of $G$-It\^{o}'s formula.

\begin{theorem}[\cite{Lin}]\label{Gito}
Let $\Phi\in C^2(\mathbb{R})$ be a real-valued function such that $\frac{d^2\Phi}{dx^2}$ satisfies the polynomial growth condition. For each $0\leq s\leq t\leq T$, consider the process $X$ with dynamics
\begin{align}\label{dynamic}
X_t=X_s+\int_s^t f_u du+\int_s^t h_ud\langle B\rangle_u+\int_s^t g_u dB_u+K_t-K_s,
\end{align}
where $f,h,g$ are bounded processes in $M_G^2(0,T)$ and $K\in M_I(0,T)\cap M_G^2(0,T)$ satisfies 
\begin{align}\label{lpestimate}
\hat{\mathbb{E}}[K_T^p]<\infty \textrm{ for any } p>2
\end{align}
and for each $t\in[0,T]$, 
\begin{align}\label{continuityproperty}
\lim_{s\rightarrow t}\hat{\mathbb{E}}[|K_t-K_s|^2]=0.
\end{align}
Then, we have
\begin{align*}
\Phi(X_t)-\Phi(X_s)=&\int_s^t \frac{d\Phi}{dx} (X_u)f_u du+\int_s^t \frac{d\Phi}{dx} (X_u)h_u d\langle B\rangle_u+\int_s^t \frac{d\Phi}{dx} (X_u)g_u dB_u\\
&+\int_s^t \frac{d\Phi}{dx} (X_u)dK_u+\frac{1}{2}\int_s^t \frac{d^2\Phi}{dx^2} (X_u)g^2_u d\langle B\rangle_u.
\end{align*}
\end{theorem}

\begin{remark}
(i) Suppose that $K\in M_c(0,T)$ is of bounded variation, q.s. Then,  the integral $\int_0^T X_tdK_t$ can be defined similarly as in Definition \ref{deflin}. Besides, similar result as in Theorem \ref{Gito} still holds when $K$ in the dynamic \eqref{dynamic} is replaced by a bounded variation process (see Remark 3.9 and Remark 3.19 in \cite{Lin}).

\noindent (ii) Suppose that $K\in M_I(0,T)$ and $K_t\in L_G^2(\Omega_t)$ for any $t\in[0,T]$. We may drop the condition \eqref{continuityproperty}. In fact, since $|K_t-K_s|\downarrow 0$ as $s$ approaches $t$, by Theorem 6.1.35 in \cite{P19}, we obtain that 
\begin{align*}
\lim_{s\rightarrow t}\hat{\mathbb{E}}[|K_t-K_s|^2]=0.
\end{align*}

\noindent (iii) Suppose that $|\frac{d^2\Phi}{dx^2}(x)|\leq C(1+|x|^k)$, for some $k\geq 1$. Condition \eqref{lpestimate} can be weakened to $\hat{\mathbb{E}}[|K_T|^{2(k+3)}]<\infty$ (see Remark 3.18 in \cite{Lin}). 
\end{remark}

\subsection{Skorokhod problem with two nonlinear constraints}

In this subsection, we study the Skorokhod problem with two nonlinear constraints, where the nonlinear functions are not assumed to be bi-Lipschitz. Let $D[0,\infty)$ be the set of real-valued  c\`{a}dl\`{a}g functions. Similarly, we use $I[0,\infty)$, $C[0,\infty)$, $BV[0,\infty)$ to denote subspaces of $D[0,\infty)$ consisting of nondecreasing functions, continuous functions and functions with bounded variation, respectively. 

\begin{definition}\label{def'}
Let $s\in D[0,\infty)$, $g,h:[0,\infty)\times \mathbb{R}\rightarrow \mathbb{R}$ be two functions with $g\leq h$. A pair of functions $(x,k)\in D[0,\infty)\times BV[0,\infty)$ is called a solution of the Skorokhod problem for $s$ with nonlinear constraints $g,h$ ($(x,k)=\mathbb{SP}_g^h(s)$ for short) if 
\begin{itemize}
\item[(i)] $x_t=s_t+k_t$.
\item[(ii)] $g(t,x_t)\leq 0\leq h(t,x_t)$.
\item[(iii)] $k_{0-}=0$ and $k$ has the decomposition $k=k^h-k^g$, where $k^h,k^g$ are nondecreasing functions satisfying
\begin{align*}
\int_0^\infty I_{\{g(s,x_s)<0\}}dk^g_s=0, \  \int_0^\infty I_{\{h(s,x_s)>0\}}dk^h_s=0.
\end{align*}
\end{itemize}
The pair $(k^g,k^h)$ will be referred to as the constraining processes associated with the Skorokhod problem $\mathbb{SP}_g^h(s)$.
\end{definition}

In order to establish the existence and uniqueness result for the Skorokhod problem $\mathbb{SP}_g^h(s)$, we propose the following assumption on the functions $g,h$.
\begin{assumption}\label{ass}
The functions $g,h:[0,\infty)\times \mathbb{R}\rightarrow \mathbb{R}$ satisfy the following conditions
\begin{itemize}
\item[(i)] For each fixed $x\in\mathbb{R}$, $g(\cdot,x),h(\cdot,x)\in D[0,\infty)$.
\item[(ii)] For any fixed $t\geq 0$, $g(t,\cdot)$, $h(t,\cdot)$ are strictly increasing and $g(t,x)< h(t,x)$ for any $(t,x)\in[0,\infty)\times\mathbb{R}$.
\item[(iii)] For any $t\geq 0$, the inverse $g^{-1}(t,0)$ and $h^{-1}(t,0)$ exist and $g^{-1}(\cdot,0),h^{-1}(\cdot,0)\in D[0,\infty)$. Moreover,  
\begin{align*}
\inf_{t\geq 0}(g^{-1}(t,0)-h^{-1}(t,0))>0.
\end{align*}
\end{itemize}
\end{assumption}

\begin{remark}\label{remark1'}
(1) Compared with Assumption 2.1 made for the functions $g,h$ of the Skorokhod problem in \cite{Li}, Assumption \ref{ass} described above is weaker. More precisely, Assumption 2.1 in \cite{Li} requires (i), (ii) in Assumption \ref{ass} and the following bi-Lipschitz condition and separated condition
\begin{itemize}
\item[(iii')] There exists two positive constants $0<c<C<\infty$, such that for any $t\geq 0$ and $x,y\in \mathbb{R}$,
\begin{align*}
&c|x-y|\leq |g(t,x)-g(t,y)|\leq C|x-y|,\\
&c|x-y|\leq |h(t,x)-h(t,y)|\leq C|x-y|.
\end{align*} 
\item[(iv')] $\inf_{(t,x)\in[0,\infty)\times\mathbb{R}}(h(t,x)-g(t,x))>0$.
\end{itemize}
By Lemma 2.1 in \cite{Li}, under conditions (iii') and (iv'), the functions $g,h$ naturally fulfill condition (iii) in Assumption \ref{ass}.

\noindent (2) The most frequently used $g,h$ take the following form
\begin{align*}
g(t,x)=x-\beta_t, \ h(t,x)=x-\alpha_t,
\end{align*}
where $\alpha,\beta\in D[0,\infty)$ with $\inf_{t\geq 0}(\beta_t-\alpha_t)>0$. In this case, the Skorokhod problem $\mathbb{SP}_g^h(s)$ degenerates to the Skorokhod problem on $[\alpha,\beta]$ for $s$ as in \cite{BKR,S1,S2}. Another example satisfying Assumption \ref{ass} is that 
\begin{align*}
g(t,x)=\rho(x)-\beta_t, \ h(t,x)=\rho(x)-\alpha_t,
\end{align*}
with $\alpha,\beta$ are given as above and 
\begin{align*}
\rho(x)=\begin{cases}
\sqrt{x}, &x\geq0,\\
-\sqrt{-x}, &x<0.
\end{cases}
\end{align*}
However, this example does not satisfy Assumption 2.1 in \cite{Li}.
\end{remark}

Under Assumption \ref{ass}, given $s\in D[0,\infty)$, for any $t\geq 0$, each of the following equation admits a unique solution
\begin{align}
g(t,s_t+x)=0, \ h(t,s_t+x)=0,
\end{align}
which is denoted by $\phi_t^s$, $\psi_t^s$, respectively. For simplicity, we omit the superscript $s$. Actually, we have
\begin{align}\label{phipsi}
\phi_t=g^{-1}(t,0)-s_t, \ \psi_t=h^{-1}(t,0)-s_t.
\end{align}
It follows from Assumption \ref{ass} that $\phi,\psi\in D[0,\infty)$ and 
\begin{align}\label{separated}
		\inf_{t\geq 0} (\phi_t-\psi_t)>0.
\end{align}

Recall that to make sure the proof of the existence and uniqueness for the Skorokhod problem in \cite{Li} valid, we need the c\`{a}dl\`{a}g property and the separated property \eqref{separated} for $\phi,\psi$, the strictly increasing property for the functions $g,h$ and the fact that $g(t,x)< h(t,x)$ for any $(t,x)\in[0,\infty)\times\mathbb{R}$. These properties naturally hold under Assumption \ref{ass}, which indicates the following result.

\begin{theorem}\label{thm1.1}
Suppose that $g,h$ satisfy Assumption \ref{ass}. For any given $s\in D[0,\infty)$, there exists a unique pair of solution to the Skorokhod problem $\mathbb{SP}_g^h(s)$. Moreover, the second component of the solution $k$ has the following representation
\begin{align*}
k_t=\min\left([-\phi^-_0]\vee \sup_{r\in[0,t]}\psi_r,\inf_{s\in[0,t]}\left[\phi_s\vee \sup_{r\in[s,t]}\psi_r\right] \right).
\end{align*}
\end{theorem}

\begin{remark}\label{remark1}
 If $s\in C[0,\infty)$ and the c\`{a}dl\`{a}g requirements in Assumption \ref{ass} are replaced by continuous requirements, then each component of the solution to the Skorokhod problem $\mathbb{SP}_g^h(s)$ is continuous. In this case, (iii) of Definition \ref{def'} can be written as 
\begin{align*}
\int_0^\infty g(s,x_s)dk^g_s=0, \  \int_0^\infty h(s,x_s)dk^h_s=0.
\end{align*}
\end{remark}

Modifying the proof of Theorem 3.2 and Proposition 3.3 in \cite{Li}, we obtain some monotonicity results of the constraining processes with respect to the nonlinear constraints $g,h$ and the input $s$, respectively. 

\begin{proposition}\label{proposition3.3}
Suppose $g^i,h^i$ satisfy Assumption \ref{ass}, $i=1,2$ with $h^1\geq h^2$ and $g^1\leq g^2$. For any given $s\in D[0,\infty)$, let $(x^i,k^i)$ be the solution to the Skorokhod problem $\mathbb{SP}_{g^i}^{h^i}(s)$ with $k^i=k^{i,h}-k^{i,g}$. Then, for any $t\geq 0$, we have
\begin{align*}
k^{2,h}_t\geq k^{1,h}_t \textrm{ and } k^{2,g}_t\geq k^{1,g}_t.
\end{align*}
\end{proposition}

\begin{proposition}\label{proposition3.4}
Let Assumption \ref{ass} hold. Given $c^i_0\in\mathbb{R}$ and $s^i\in D[0,\infty)$ with $s^1_0=s^2_0=0$, $i=1,2$, suppose that there exists $\nu\in I[0,\infty)$ such that $s^1=s^2+\nu$. Let $(x^i,k^i)=\mathbb{SP}_g^h(c_0^i+s^i)$ with decomposition $k^i=k^{i,h}-k^{i,g}$, $i=1,2$. Then, we have
\begin{itemize}
\item[1.] $k_t^{1,h}-(c^2_0-c^1_0)^+\leq k^{2,h}_t\leq k^{1,h}_t+\nu_t+(c_0^1-c_0^2)^+$;
\item[2.] $k_t^{2,g}-(c^2_0-c^1_0)^+\leq k^{1,g}_t\leq k^{2,g}_t+\nu_t+(c_0^1-c_0^2)^+$.
\end{itemize}
\end{proposition}


The following proposition provides the continuity property of the solution to the Skorokhod problem with respect to input function and constraint functions, which is slightly different from Proposition 3.4 in \cite{Li}.
\begin{proposition}\label{prop4.1}
Suppose that $(g^i,h^i)$ satisfy Assumption \ref{ass}, $i=1,2$. Given $s^i\in D[0,\infty)$, let $(x^i,k^i)$ be the solution to the Skorokhod problem $\mathbb{SP}_{g^i}^{h^i}(s^i)$. Then, we have
\begin{align*}
\sup_{t\in[0,T]}|k^1_t-k^2_t|\leq \sup_{t\in[0,T]}|s^1_t-s^2_t|+\sup_{t\in[0,T]}(\bar{g}_t\vee\bar{h}_t),
\end{align*}
where
\begin{align*}
&\bar{g}_t:=|(g^1)^{-1}(t,0)-(g^2)^{-1}(t,0)|,\\
&\bar{h}_t:=|(h^1)^{-1}(t,0)-(h^2)^{-1}(t,0)|.
\end{align*}
\end{proposition}

\begin{proof}
By a similar analysis as the proof of Proposition 3.4 in \cite{Li} and the representation for the second component of the solution to the Skorokhod problem obtained in Theorem \ref{thm1.1}, we have for any $t\in[0,T]$,
\begin{align*}
|k^1_t-k^2_t|\leq \sup_{s\in[0,T]}|\phi^1_s-\phi^2_s|\vee \sup_{s\in[0,T]}|\psi^1_s-\psi^2_s|.
\end{align*}
Recalling \eqref{phipsi}, we obtain that
\begin{align*}
\phi^i_t=(g^i)^{-1}(t,0)-s^i_t, \ \psi^i_t=(h^i)^{-1}(t,0)-s^i_t, \ i=1,2.
\end{align*}
All the above analysis indicates the desired result. 
\end{proof}

	\section{Reflected $G$-Brownian motion with nonlinear constraints}
	
	In this section, we first consider the reflected $G$-Brownian motion with nonlinear constraints on a fixed time horizon $[0,T]$. Actually, this is a special case of reflected $G$-SDE \eqref{eq1.1} by taking $x=f=h=0$ and $g=1$. We propose the following assumption on the nonlinear constraints $l,r$.
	
	\begin{assumption}\label{ass1}
The functions $r,l:[0,T]\times \Omega_T\times \mathbb{R}\rightarrow \mathbb{R}$ satisfy the following conditions:
\begin{itemize}
\item[(i)] For each fixed $x\in\mathbb{R}$, $r(\cdot,x),l(\cdot,x)\in C[0,\infty)$, q.s.
\item[(ii)] For any fixed $t\geq 0$, $r(t,\cdot)$, $l(t,\cdot)$ are strictly increasing  and $r(t,x)>l(t,x)$ for any fixed $(t,x)\in[0,\infty)\times\mathbb{R}$, q.s.
\item[(iii)] For any $t\geq 0$, the inverse $r^{-1}(t,0)$ and $l^{-1}(t,0)$ exist, q.s. and $\{r^{-1}(t,0)\}_{t\in[0,T]}$, $\{l^{-1}(t,0)\}_{t\in[0,T]}$ belong to $S_G^p(0,T)$ for some $p\geq 2$. Moreover, 
\begin{align*}
\inf_{(t,\omega)\in[0,T]\times\Omega_T}(l^{-1}(t,\omega,0)-r^{-1}(t,\omega,0))>0.
\end{align*}
\end{itemize}
\end{assumption}

\begin{remark}\label{remark2}
Suppose that $l,r$ take the following form
\begin{align*}
l(t,x)=\rho(x)-\beta_t, \ r(t,x)=\rho(x)-\alpha_t,
\end{align*}
where $\rho$ is the same as in Remark \ref{remark1'} and  $\alpha,\beta\in S_G^{2p}(0,T)$ satisfy $\inf_{(t,\omega)\in [0,T]\times\Omega_T}(\beta_t(\omega)-\alpha_t(\omega))>0$. It is easy to check that $r^{-1}(t,0)=\rho^{-1}(\alpha_t)$ and $l^{-1}(t,0)=\rho^{-1}(\beta_t)$, where 
\begin{align*}
\rho^{-1}(x)=\begin{cases}
x^2, &x\geq0,\\
-x^2, &x<0.
\end{cases}
\end{align*}
 Then, $l,r$ fulfill the conditions in Assumption \ref{ass1}.
\end{remark}

	Before introducing the main result, we first give the following lemma.

\begin{lemma}\label{SGP}
Suppose that $\alpha,\beta\in S_G^p(0,T)$ for some $p\geq 1$. Define the process $\gamma$ by
\begin{align*}
\gamma_t:=\sup_{s\in[0,t]}\left[\alpha_s\wedge (\inf_{r\in[s,t]}\beta_r)\right].
\end{align*}
Then, we have $\gamma\in S_G^p(0,T)$.
\end{lemma}

\begin{proof}
For any fixed positive integer $n$, let $t^n_i=\frac{iT}{n}$, $i=0,1,\cdots,n$. For simplicity, we always omit the superscript $n$. Set 
\begin{align*}
\gamma^n_t:=\vee_{j=0}^{n}\left[\alpha_{t\wedge t_j}\wedge (\wedge_{i=j}^n \beta_{t\wedge t_i})\right].
\end{align*}
Then, we have $\gamma^n\in S_G^p(0,T)$. For any $t\in[t_k,t_{k+1})$, $k=0,1,\cdots,n-1$, it is easy to check that 
\begin{align*}
\gamma_t=\left(\vee_{j=0}^{k} \sup_{s\in[t_j,t_{j+1})}\left[\alpha_s\wedge \inf_{r\in [s,t_{j+1})}\beta_r\wedge \left(\wedge_{i=j+1}^{k-1} \inf_{r\in[t_i,t_{i+1})}\beta_r\right)\wedge \inf_{r\in[t_k,t]}\beta_r\right]\right)
\vee \sup_{s\in[t_k,t]}\left[\alpha_s\wedge \inf_{r\in[s,t]}\beta_r\right]
\end{align*}
and 
\begin{align*}
\gamma^n_t=\left(\vee_{j=0}^{k}\left[\alpha_{t_j}\wedge (\wedge_{i=j}^k \beta_{t_i})\wedge \beta_t\right]\right)\vee [\alpha_t\wedge \beta_t].
\end{align*}
Consequently, for $t\in[t_k,t_{k+1})$, we have
\begin{align*}
|\gamma_t-\gamma^n_t|\leq &\left(\vee_{j=0}^k\sup_{s\in[t_j,t_{j+1})}\left[|\alpha_s-\alpha_{t_j}|\vee I_{s,t}^n(\beta)\right]\right)
\vee \sup_{s\in[t_k,t]}\left[|\alpha_s-\alpha_t|\vee \sup_{r\in[s,t]}|\beta_r-\beta_t|\right],
\end{align*}
where
\begin{align*}
I_{s,t}^n(\beta)=\sup_{r\in [s,t_{j+1})}|\beta_r-\beta_{t_i}| \vee \sup_{r\in[t_k,t]}|\beta_r-\beta_{t_k}|\vee \sup_{r\in[t_k,t]}|\beta_r-\beta_{t}|\vee_{i=j+1}^{k-1}\sup_{r\in[t_i,t_{i+1})}|\beta_r-\beta_{t_i}|.
\end{align*}
Then, there exists a constant $C$ depending on $p$, such that 
\begin{align*}
\hat{\mathbb{E}}[\sup_{t\in[0,T]}|\gamma_t-\gamma^n_t|^p]\leq C(\hat{\mathbb{E}}[\sup_{i=0,1,\cdots,n-1}\sup_{s\in[t_i,t_{i+1}]}|\alpha_s-\alpha_{t_i}|^p]+\hat{\mathbb{E}}[\sup_{i=0,1,\cdots,n-1}\sup_{s\in[t_i,t_{i+1}]}|\beta_s-\beta_{t_i}|^p]).
\end{align*}
By Proposition \ref{the3.7}, we have
\begin{align*}
\lim_{n\rightarrow\infty}\hat{\mathbb{E}}[\sup_{t\in[0,T]}|\gamma_t-\gamma^n_t|^p]=0,
\end{align*}
which yields the desired result.
\end{proof}

		\begin{definition}
	We denote by $M_{BV}(0,T)$ the collection of all q.s. continuous processes $X$ whose paths $X_\cdot(\omega):[0,T]\rightarrow \mathbb{R}$ are of bounded variation outside a polar set $A$.
	\end{definition}	
	
	Based on Theorem \ref{thm1.1}, we may construct the doubly reflected $G$-Brownian motion with nonlinear constraints $l,r$. 
	
	\begin{theorem}\label{thm4.1}
	Let Assumption \ref{ass1} hold. 
	Then, there exists a unique pair of processes $(X,A)\in S_G^p(0,T)\times (M_{BV}(0,T)\cap S_G^p(0,T))$, such that 
	\begin{displaymath}
	X_t=B_t+A_t, \ t\in[0,T], \textrm{ q.s.,}
	\end{displaymath}
	where  
	\begin{itemize}
	\item[(a)] $l(t,X_t)\leq 0\leq r(t,X_t)$, $t\in[0,T]$, q.s.;
	\item[(b)] $A_{0-}=0$ and $A$ has the decomposition $A=A^r-A^l$, where $A^r,A^l\in M_I(0,T)\cap S_G^p(0,T)$ and
\begin{align*}
\int_0^T l(s,X_s)dA^l_s=\int_0^T r(s,X_s)dA^r_s=0, \textrm{ q.s.}
\end{align*}
\end{itemize}
	\end{theorem}

\begin{proof}
Without loss of generality, assume that for any $\omega\in \Omega_T$, (i)-(iii) in Assumption \ref{ass1} hold. For any $\omega\in \Omega_T$, set 
\begin{align}\label{lomegaromega}
l_\omega(t,x):=l(t,\omega,x), \ r_\omega(t,x):=r(t,\omega,x).
\end{align}
 It is easy to check that $l_\omega,r_\omega$ satisfy Assumption \ref{ass}. By Theorem \ref{thm1.1}, the Skorokhod problem $\mathbb{SP}_{l_\omega}^{r_\omega}(B(\omega))$ admits a unique pair of solution, denoted by $(X(\omega),A(\omega))$. It remains to prove that $X,A\in S_G^p(0,T)$. 
By Theorem \ref{thm1.1}, $A$ can be represented as follows
\begin{align*}
A_t=\min\left([-\phi^-_0]\vee \sup_{r\in[0,t]}\psi_r,\inf_{s\in[0,t]}\left[\phi_s\vee \sup_{r\in[s,t]}\psi_r\right] \right),
\end{align*}
where 
\begin{align*}
\phi_t=l^{-1}(t,0)-B_t, \ \psi_t=r^{-1}(t,0)-B_t.
\end{align*}
Clearly, $\phi,\psi\in S_G^p(0,T)$. By Lemma \ref{SGP}, we obtain that $A\in S_G^p(0,T)$. The proof is complete.
\end{proof}

By the proof of Theorem \ref{thm4.1}, similar result still holds if the $G$-Brownian motion $B$ is replaced by some $S\in S_G^p(0,T)$.

\begin{theorem}\label{thm4.3}
Let Assumption \ref{ass1} hold. Given $S\in S_G^p(0,T)$, then, there exists a unique pair of processes $(X,A)\in S_G^p(0,T)\times (M_{BV}(0,T)\cap S_G^p(0,T))$, such that 
	\begin{displaymath}
	X_t=S_t+A_t, \ t\in[0,T], \textrm{ q.s.,}
	\end{displaymath}
	where  
	\begin{itemize}
	\item[(a)] $l(t,X_t)\leq 0\leq r(t,X_t)$, $t\in[0,T]$, q.s.;
	\item[(b)] $A_{0-}=0$ and $A$ has the decomposition $A=A^r-A^l$, where $A^r,A^l\in M_I(0,T)\cap S_G^p(0,T)$ and
\begin{align*}
\int_0^T l(s,X_s)dA^l_s=\int_0^T r(s,X_s)dA^r_s=0, \textrm{ q.s.}
\end{align*}
\end{itemize}
\end{theorem}

\begin{remark}
(i) The pair $(X,A)$ in Theorem \ref{thm4.3} is called the solution of the reflected problem with constraints $l,r$ for $S$. Moreover, the pair $(A^r,A^l)$ will be referred to as the constraining processes corresponding to the reflected problem with constraints $l,r$ for $S$.

\noindent (ii) If $l= -\infty$ and $r(t,x)=x$, the process $X$ in Theorem \ref{thm4.1} is called a $G$-reflected Brownian motion on the half-line $[0,\infty)$ studied in \cite{Lin}.  

\noindent (iii) Recall that the process $S$ in the reflected problem considered in \cite{Lin} is usually required to be a $G$-It\^{o} process (see Theorem 4.3 in \cite{Lin}), or, at least, satisfies the following property
\begin{align*}
\hat{\mathbb{E}}[\sup_{u\in[s,t]}|S_u-S_s|^p]\leq C|t-s|^{p/2}, \ 0\leq s<t\leq T.
\end{align*}
Compared with the result in \cite{Lin}, our assumption, i.e., the input process $S$ belong to $S_G^p(0,T)$, is weaker. Besides, each component of the solution to the reflected problem belongs to $S_G^p(0,T)$, which enhances the results in \cite{Lin} since the solution to the reflected problem in \cite{Lin} belongs to $M_G^p(0,T)$.
\end{remark}

Actually, for the reflected problem with constraints $l,r$ for $S$, the objective of $A^r$ (resp. $A^l$) is to push the solution upwards such that $r(t,X_t)\geq 0$ (resp. to pull the solution downwards such that $l(t,X_t)\leq 0$). Now, consider two reflected problems with different constraints but same input process. If the lower constraint function is larger, then the force aiming to pull the solution downwards should be larger. If the upper constraint function is smaller, then the force aiming to push the solution upwards should be larger. For two reflected problems with same constraints but different input processes, if the input process is larger, then the force aiming to push the solution upwards should be smaller while the force aiming to pull the solution downwards should be larger. More precisely, we have the following comparison results for the individual constraining processes, which are based on the monotonicity properties for the deterministic case (Proposition \ref{proposition3.3} and \ref{proposition3.4}) and the pathwise construction for solutions to the reflected problem with nonlinear constraints.

\begin{proposition}
Suppose that $(l,r)$ and $(\tilde{l},\tilde{r})$ satisfy Assumption \ref{ass1}. Given $S\in S_G^p(0,T)$, let $(A^l,A^r)$ (resp., $(\tilde{A}^l,\tilde{A}^u)$) be the pair of constraining processes for the reflected problem with constraints $l,r$ (resp., $\tilde{l},\tilde{r}$) for $S$. If $r\geq \tilde{r}$ and $l\leq \tilde{l}$, then for any $t\in[0,T]$, 
\begin{displaymath}
A^l_t\leq \tilde{A}^l_t, \  A^r_t\leq \tilde{A}^r_t, \textrm{ q.s.}
\end{displaymath}
\end{proposition}

\begin{proposition}\label{prop3.7}
Let $S,\tilde{S}\in S_G^p(0,T)$ with $S_0=\tilde{S}_0$ and Assumption \ref{ass1} hold. Given $c_0,\tilde{c}_0\in\mathbb{R}$, let $(X,A)$, $(\tilde{X},\tilde{A})$ be the solution to the doubly reflected problem with constraints $l,r$ for $c_0+S$ and $\tilde{c}_0+\tilde{S}$, respectively. Moreover, let $(A^l,A^r)$, $(\tilde{A}^l,\tilde{A}^r)$ be the corresponding constraining processes. If there exists a nondecreasing function $\nu\in S_G^p(0,T)$ with $\nu_0=0$ such that $S=\tilde{S}+\nu$, then for each $t\in [0,T]$, the following relations hold: 
\begin{itemize}
\item[1.] $A^r_t-(c_0-\tilde{c}_0)^-\leq \tilde{A}^r_t\leq A^r_t+\nu_t+(c_0-\tilde{c}_0)^+$, q.s.;
\item[2.] $\tilde{A}^l_t-(c_0-\tilde{c}_0)^-\leq {A}^l_t\leq \tilde{A}^l_t+\nu_t+(c_0-\tilde{c}_0)^+$, q.s.
\end{itemize}
\end{proposition}

\section{Reflected $G$-SDEs with nonlinear constraints}

In this section, we study the reflected SDEs driven by $G$-Brownian motion with nonlinear constraints. We are given the following parameters: the initial value $x_0$, the constraints $l,r$ and the coefficients $f,h,g:[0,T]\times\Omega_T\times\mathbb{R}\rightarrow\mathbb{R}$. We propose the following assumptions on the coefficients, which are the same as those made for the non-reflected case.

\begin{assumption}\label{ass2}
\begin{itemize}
\item[(A1)] For some $p\geq 2$ and for each $x\in\mathbb{R}$, $f(\cdot,\cdot,x),h(\cdot,\cdot,x)\in M_G^p(0,T)$ and $g(\cdot,\cdot,x)\in H_G^p(0,T)$;
\item[(A2)] For any $(t,\omega)\in[0,T]\times\Omega$ and any $x,x'\in \mathbb{R}$, there exists a constant $L>0$, such that 
$$
|f(t,\omega,x)-f(t,\omega,x')|+|g(t,\omega,x)-g(t,\omega,x')|+|h(t,\omega,x)-h(t,\omega,x')|\leq L|x-x'|.
$$
\end{itemize}
\end{assumption}

The solution to doubly reflected $G$-SDE with parameters $(x_0,f,h,g,l,r)$ is a pair of processes $(X,A)$ such that 
\begin{itemize}
\item[(1)] $X\in S_G^p(0,T)$ and $A\in S_G^p(0,T)\cap M_{BV}(0,T)$ satisfy the following equation:
\begin{equation}\label{e5.1}
X_t=x_0+\int_0^t f(s,X_s)ds+\int_0^t h(x,X_s)d\langle B\rangle_s+\int_0^t g(s,X_s)dB_s+A_t, \ t\in[0,T];
\end{equation}
\item[(2)] $l(t, X_t)\leq 0\leq r(t,X_t)$, $t\in[0,T]$, q.s.;
\item[(3)] $A_{0-}=0$ and $A$ has the decomposition $A=A^r-A^l$, where $A^r,A^l\in M_I(0,T)\cap S_G^p(0,T)$ and
\begin{align*}
\int_0^T l(s,X_s)dA^l_s=\int_0^T r(s,X_s)dA^r_s=0, \textrm{ q.s.}
\end{align*}
\end{itemize}

\subsection{Existence and uniqueness result}

Now, we state the main result in this subsection, which indicates that there exists a unique solution to the reflected $G$-SDE with nonlinear constraints.

\begin{theorem}\label{thm5.3}
Suppose that Assumptions \ref{ass1} and Assumption \ref{ass2} hold. Then, the doubly reflected $G$-SDE with parameters $(x_0,f,h,g,l,r)$ admits a unique pair of solution $(X,A)$. 
\end{theorem}

\begin{remark}
Suppose that $l=-\infty$ and $r(t,x)=x-S_t$ with $S$ being a $G$-It\^{o} process. The reflected $G$-SDE with constraints $l,r$ degenerates into the reflected $G$-SDE studied in \cite{Lin}. Another frequently used constraint functions are given by $l(t,x)=x-\beta_t$, $r(t,x)=x-\alpha_t$, where $\alpha,\beta\in S_G^p(0,T)$ satisfy
\begin{align*}
\inf_{(t,\omega)\in[0,T]\times \Omega_T}(\beta_t(\omega)-\alpha_t(\omega))>0.
\end{align*}
We write the parameters of this kind of  doubly reflected $G$-SDEs as $(x_0,f,h,g,\alpha,\beta)$.
\end{remark}

We first establish some a priori estimates, which will be helpful to derive the uniqueness result for reflected $G$-SDEs with nonlinear constraints. In the following, $C$ will always represent a positive constant, which may vary from line to line.

\begin{proposition}\label{prop5.1}
Let $(X,A)$ be the solution to doubly reflected $G$-SDE with parameters $(x_0,f,h,g,l,r)$. Then, there exists a positive constant $C$ depending on $T,G,L,p$, such that 
\begin{align*}
&\hat{\mathbb{E}}\left[\sup_{t\in[0,T]}|X_t|^p\right]+\hat{\mathbb{E}}\left[\sup_{t\in[0,T]}|A_t|^p\right]\\
\leq C&\Bigg\{|x_0|^p+\hat{\mathbb{E}}\left[\int_0^T|f(t,0)|^pdt\right]+\hat{\mathbb{E}}\left[\int_0^T|h(t,0)|^pdt\right]+\hat{\mathbb{E}}\left[\left(\int_0^T|g(t,0)|^2 dt\right)^{\frac{p}{2}}\right] \\
&+\hat{\mathbb{E}}\left[\sup_{s\in [0,T]}|l^{-1}(s,0)|^p\right]+\hat{\mathbb{E}}\left[\sup_{s\in [0,T]}|r^{-1}(s,0)|^p\right]\Bigg\}.
\end{align*}
\end{proposition}

\begin{proof}
Since $(X,A)$ satisfy \eqref{e5.1}, it is easy to check that 
\begin{equation}\label{e5.3}\begin{split}
\hat{\mathbb{E}}\left[\sup_{s\in[0,t]}|X_s|^p\right]\leq C&\Bigg\{|x_0|^p+\hat{\mathbb{E}}\left[\sup_{s\in[0,t]}\left|\int_0^s f(u,X_u)du\right|^p\right]+\hat{\mathbb{E}}\left[\sup_{s\in[0,t]}\left|\int_0^s h(u,X_u)d\langle B\rangle_u\right|^p\right]\\
&+\hat{\mathbb{E}}\left[\sup_{s\in[0,t]}\left|\int_0^s g(u,X_u)dB_u\right|^p\right]+\hat{\mathbb{E}}\left[\sup_{s\in[0,t]}|A_s|^p\right]\Bigg\}\\
\leq C&\Bigg\{|x_0|^p+\hat{\mathbb{E}}\left[\int_0^t |f(u,0)|^p du\right]+\hat{\mathbb{E}}\left[\left(\int_0^t |g(u,0)|^2du\right)^{\frac{p}{2}}\right]\\
&+\hat{\mathbb{E}}\left[\int_0^t |h(u,0)|^p du\right]+\hat{\mathbb{E}}\left[\sup_{s\in[0,t]}|A_s|^p\right]+\hat{\mathbb{E}}\left[\int_0^t |X_s|^p ds\right]\Bigg\}\\
\leq C&\Bigg\{|x_0|^p+\hat{\mathbb{E}}\left[\int_0^t |f(u,0)|^p du\right]+\hat{\mathbb{E}}\left[\left(\int_0^t |g(u,0)|^2du\right)^{\frac{p}{2}}\right]\\
&+\hat{\mathbb{E}}\left[\int_0^t |h(u,0)|^p du\right]+\hat{\mathbb{E}}\left[\sup_{s\in[0,t]}|A_s|^p\right]+\int_0^t \hat{\mathbb{E}}\left[\sup_{u\in[0,s]}|X_u|^p \right]du\Bigg\},
\end{split}\end{equation}
where we have used Proposition \ref{the1.3} and the Lipschitz continuity of $f,g,h$ in the second inequality. Set 
\begin{displaymath}
\mathbb{X}_t=x_0+\int_0^t f(s,X_s)ds+\int_0^t h(s,X_s)d\langle B\rangle_s+\int_0^t g(s,X_s)dB_s.
\end{displaymath}
By the proof of Theorem \ref{thm4.1}, $A(\omega)$ can be regarded as the second component of the solution to the Skorokhod problem for $\mathbb{X}(\omega)$ with nonlinear constraints $l_\omega,r_\omega$, where $l_\omega,r_\omega$ is defined in \eqref{lomegaromega}. Let $(X',A')$ be the solution of the reflected problem with constraints $l,r$ for $0$. By the proof of Theorem \ref{thm4.1} again, we have 
\begin{align*}
A'_t=\min\left([-(l^{-1}(0,0))^-]\vee \sup_{u\in[0,t]}r^{-1}(u,0),\inf_{s\in[0,t]}\left[l^{-1}(s,0)\vee \sup_{u\in[s,t]}r^{-1}(u,0)\right] \right).
\end{align*} 
It follows that there exists a constant $C$, such that 
\begin{align*}
\hat{\mathbb{E}}\left[\sup_{s\in [0,t]}|A'_s|^p\right]\leq C\left\{\hat{\mathbb{E}}\left[\sup_{s\in [0,t]}|l^{-1}(s,0)|^p\right]+\hat{\mathbb{E}}\left[\sup_{s\in [0,t]}|r^{-1}(s,0)|^p\right]\right\}.
\end{align*}
Applying Proposition \ref{prop4.1} yields that 
\begin{equation}\label{e5.4}\begin{split}
\hat{\mathbb{E}}\left[\sup_{s\in[0,t]}|A_s|^p\right]=&\hat{\mathbb{E}}\left[\sup_{s\in[0,t]}|A_s-A'_s+A'_s|^p\right]\\
\leq &C\left\{\hat{\mathbb{E}}\left[\sup_{s\in[0,t]}|A'_s|^p\right]+\hat{\mathbb{E}}\left[\sup_{s\in[0,t]}|\mathbb{X}_s|^p\right]\right\}\\
\leq &C\Bigg\{|x_0|^p+\hat{\mathbb{E}}\left[\int_0^t |f(u,0)|^p du\right]+\hat{\mathbb{E}}\left[\left(\int_0^t |g(u,0)|^2du\right)^{\frac{p}{2}}\right]+\hat{\mathbb{E}}\left[\int_0^t |h(u,0)|^p du\right]\\
&+\int_0^t \hat{\mathbb{E}}\left[\sup_{u\in[0,s]}|X_u|^p \right]du+\hat{\mathbb{E}}\left[\sup_{s\in [0,t]}|l^{-1}(s,0)|^p\right]+\hat{\mathbb{E}}\left[\sup_{s\in [0,t]}|r^{-1}(s,0)|^p\right]\Bigg\}.
\end{split}\end{equation}
Combining \eqref{e5.3} and \eqref{e5.4}, we obtain that
\begin{align*}
\hat{\mathbb{E}}\left[\sup_{s\in[0,t]}|X_s|^p\right]\leq &C\Bigg\{|x_0|^p+\hat{\mathbb{E}}\left[\int_0^T |f(u,0)|^p du\right]+\hat{\mathbb{E}}\left[\left(\int_0^T |g(u,0)|^2du\right)^{\frac{p}{2}}\right]+\hat{\mathbb{E}}\left[\int_0^T |h(u,0)|^p du\right]\\
&+\int_0^t \hat{\mathbb{E}}\left[\sup_{u\in[0,s]}|X_u|^p \right]du+\hat{\mathbb{E}}\left[\sup_{s\in [0,T]}|l^{-1}(s,0)|^p\right]+\hat{\mathbb{E}}\left[\sup_{s\in [0,T]}|r^{-1}(s,0)|^p\right]\Bigg\}.
\end{align*}
Applying the Gronwall inequality indicates that 
\begin{align*}
\hat{\mathbb{E}}\left[\sup_{s\in[0,T]}|X_s|^p\right]\leq &C\Bigg\{|x_0|^p+\hat{\mathbb{E}}\left[\int_0^T |f(u,0)|^p du\right]+\hat{\mathbb{E}}\left[\left(\int_0^T |g(u,0)|^2du\right)^{\frac{p}{2}}\right]\\ &+\hat{\mathbb{E}}\left[\int_0^T |h(u,0)|^p du\right]
+\hat{\mathbb{E}}\left[\sup_{s\in [0,T]}|l^{-1}(s,0)|^p\right]+\hat{\mathbb{E}}\left[\sup_{s\in [0,T]}|r^{-1}(s,0)|^p\right]\Bigg\}.
\end{align*}
Plugging the above result into \eqref{e5.4}, we obtain the estimate for $A$. The proof is complete.
\end{proof}

\begin{proposition}\label{prop5.2}
Let $(x^i,f^i,h^i,g^i,l^i,r^i)$, $i=1,2$ be two sets of parameters satisfying Assumption \ref{ass1} and Assumption \ref{ass2}. Suppose that $(X^i,A^i)$ is the solution to reflected $G$-SDE with parameters $(x^i,f^i,h^i,g^i,l^i,r^i)$, $i=1,2$. Set 
\begin{align*}
\hat{x}:=x^1-x^2, \hat{f}:=f^1-f^2, \hat{h}:=h^1-h^2, \hat{g}:=g^1-g^2, \hat{X}:=X^1-X^2.
\end{align*}
Then, there exists a positive constant $C$ depending on $T,G,L,p$, such that
\begin{align*}
\hat{\mathbb{E}}\left[\sup_{t\in[0,T]}|\hat{X}_t|^p\right]\leq C&\Bigg\{|\hat{x}|^p+\hat{\mathbb{E}}\left[\int_0^T|\hat{f}(t,X^1_t)|^pdt\right]+\hat{\mathbb{E}}\left[\int_0^T |\hat{h}(t,X^1_t)|^pdt\right]\\
&+\hat{\mathbb{E}}\left[\left(\int_0^T|\hat{g}(t,X^1_t)|^2 dt\right)^{\frac{p}{2}}\right]+\hat{\mathbb{E}}[|\bar{l}_T|^p]+\hat{\mathbb{E}}[|\bar{r}_T|^p]\Bigg\},
\end{align*}
where 
\begin{align*}
\bar{l}_T=\sup_{t\in[0,T]}|(l^1)^{-1}(t,0)-(l^2)^{-1}(t,0)|, \  \bar{r}_T=\sup_{t\in[0,T]}|(r^1)^{-1}(t,0)-(r^2)^{-1}(t,0)|.
\end{align*}
\end{proposition}

\begin{proof}
 Set $\hat{\mathbb{X}}:=\mathbb{X}^1-\mathbb{X}^2$, where
\begin{displaymath}
\mathbb{X}^i_t=x^i+\int_0^t f^i(s,X^i_s)ds+\int_0^t h^i(s,X^i_s)d\langle B\rangle_s+\int_0^t g^i(s,X^i_s)dB_s, \ i=1,2.
\end{displaymath}
By Proposition \ref{prop4.1} and the construction in the proof of Theorem \ref{thm4.1}, we have
\begin{equation}\label{e5.7}
\hat{\mathbb{E}}\left[\sup_{s\in[0,t]}|\hat{A}_s|^p\right]
\leq C\left\{\hat{\mathbb{E}}[|\bar{l}_T|^p]+\hat{\mathbb{E}}[|\bar{r}_T|^p]+\hat{\mathbb{E}}\left[\sup_{s\in[0,t]}|\hat{\mathbb{X}}_s|^p\right]\right\}.
\end{equation}
Simple calculation yields that 
\begin{align*}
\hat{\mathbb{E}}\left[\sup_{s\in[0,t]}|\hat{\mathbb{X}}_s|^p\right]\leq &C\Bigg\{|\hat{x}|^p+\hat{\mathbb{E}}\left[\int_0^t|\hat{f}(s,X^1_s)|^pds\right]+\hat{\mathbb{E}}\left[\left(\int_0^t|\hat{g}(s,X^1_s)|^2 ds\right)^{\frac{p}{2}}\right]\\
&+\hat{\mathbb{E}}\left[\int_0^t|\hat{h}(s,X^1_s)|^pds\right]+\int_0^t \hat{\mathbb{E}}[|\hat{X}_s|^p]ds\Bigg\}.
\end{align*}
All the above analysis indicates that 
\begin{align*}
\hat{\mathbb{E}}\left[\sup_{s\in[0,t]}|\hat{X}_s|^p\right]=&\hat{\mathbb{E}}\left[\sup_{s\in[0,t]}|\hat{\mathbb{X}}_s+\hat{A}_s|^p\right]\\
\leq &C\Bigg\{|\hat{x}|^p+\hat{\mathbb{E}}\left[\int_0^T|\hat{f}(s,X^1_s)|^pds\right]+\hat{\mathbb{E}}\left[\left(\int_0^T|\hat{g}(s,X^1_s)|^2 ds\right)^{\frac{p}{2}}\right]\\
&+\hat{\mathbb{E}}\left[\int_0^T|\hat{h}(s,X^1_s)|^pds\right]+\hat{\mathbb{E}}[|\bar{l}_T|^p]+\hat{\mathbb{E}}[|\bar{r}_T|^p]+\int_0^t \hat{\mathbb{E}}\left[\sup_{u\in[0,s]}|\hat{X}_u|^p\right]ds\Bigg\}.
\end{align*}
The desired result follows from Gronwall's inequality.
\end{proof}

Now, we are in a position to prove the main result in this subsection.

\begin{proof}[Proof of Theorem \ref{thm5.3}]
The uniqueness follows easily from Proposition \ref{prop5.2} and Equation \eqref{e5.7}. It remains to prove the existence. The proof will be based on a Picard iteration method. To this end, set $X^0\equiv 0$. For each $n\in \mathbb{N}$, let $(X^{n+1},A^{n+1})$ be the solution to the reflected problem with constraints $l,r$ for $S^n$, where 
\begin{displaymath}
S^n_t=x_0+\int_0^t f(s,X^n_s)ds+\int_0^t h(x,X^n_s)d\langle B\rangle_s+\int_0^t g(s,X^n_s)dB_s, \ t\in[0,T].
\end{displaymath}
In order to make sure the existence of $(X^{n+1},A^{n+1})$, by Theorem \ref{thm4.3}, we need to show that $S^n\in S_G^p(0,T)$. It is sufficient to prove that $\hat{\mathbb{E}}[\sup_{s\in[0,T]}|X^n_s|^p]<\infty$. We claim that for any $n\in\mathbb{N}$, 
\begin{equation}\label{claim}
\hat{\mathbb{E}}\left[\sup_{s\in[0,t]}|X^n_s|^p\right]\leq p(t),  \ t\in[0,T],
\end{equation}
where $p(\cdot)$ is the solution to the following ordinary differential equation
\begin{align*}
p(t)=&C\Bigg\{|x_0|^p+\hat{\mathbb{E}}\left[\int_0^T|f(t,0)|^pdt\right]+\hat{\mathbb{E}}\left[\int_0^T|h(t,0)|^pdt\right] +\hat{\mathbb{E}}\left[\left(\int_0^T|g(t,0)|^2 dt\right)^{\frac{p}{2}}\right]\\
&+\hat{\mathbb{E}}\left[\sup_{t\in[0,T]}|l^{-1}(t,0)|^p\right]+\hat{\mathbb{E}}\left[\sup_{t\in[0,T]}|r^{-1}(t,0)|^p\right]+\int_0^t p(s)ds\Bigg\},
\end{align*}
and $C$ is the positive constant in \eqref{e5.7'} below. Clearly, \eqref{claim} holds for $n=0$. Suppose it hold when $n=k$. Hence, the pair of processes $(X^{k+1},A^{k+1})$ exists. Following the procedures in the proof of  Proposition \ref{prop5.1}, we have
\begin{equation}\label{e5.7'}\begin{split}
\hat{\mathbb{E}}\left[\sup_{s\in[0,t]}|X^{k+1}_s|^p\right]\leq &C\Bigg\{|x_0|^p+\hat{\mathbb{E}}\left[\int_0^T|f(t,0)|^pdt\right]+\hat{\mathbb{E}}\left[\int_0^T|h(t,0)|^pdt\right] +\hat{\mathbb{E}}\left[\left(\int_0^T|g(t,0)|^2 dt\right)^{\frac{p}{2}}\right]\\
&+\hat{\mathbb{E}}\left[\sup_{t\in[0,T]}|l^{-1}(t,0)|^p\right]+\hat{\mathbb{E}}\left[\sup_{t\in[0,T]}|r^{-1}(t,0)|^p\right]+\int_0^t \hat{\mathbb{E}}\left[\sup_{u\in[0,s]}|X^k_s|^p\right]ds\Bigg\},
\end{split}\end{equation}
which implies that 
\begin{displaymath}
\hat{\mathbb{E}}\left[\sup_{s\in[0,t]}|X^{k+1}_s|^p\right]\leq p(t),  \ t\in[0,T].
\end{displaymath}
Hence, the estimate \eqref{claim} holds true.

Now, for each $m,n\in\mathbb{N}$, we define
\begin{displaymath}
u_t^{n+1,m}:=\hat{\mathbb{E}}\left[\sup_{s\in[0,t]}|X^{n+m+1}_s-X^{n+1}_s|^p\right], \ t\in[0,T].
\end{displaymath}
By a similar analysis as the proof of Proposition \ref{prop5.2}, we have
\begin{displaymath}
u^{n+1,m}_t\leq C\int_0^t u^{n,m}_s ds.
\end{displaymath} 
Set $v_t^n:=\sup_{m\in\mathbb{N}}u_t^{n,m}$, $t\in[0,T]$. It is obvious that 
\begin{displaymath}
0\leq u^{n+1,m}_t\leq  C\sup_{m\in\mathbb{N}}\int_0^t u^{n,m}_s ds\leq C\int_0^t v^n_sds.
\end{displaymath}
It follows by taking supremum over all $m\in\mathbb{N}$ on the left-hand side that
\begin{displaymath}
0\leq v_t^n\leq C\int_0^t v_s^n ds.
\end{displaymath}
We define $V_t:=\limsup_{n\rightarrow \infty}v_t^n$, $t\in[0,T]$. By the estimate \eqref{claim}, there exists a constant $C$ independent of $n$, such that $v_t^n\leq Cp(t)$. Applying Fatou's lemma, we finally have
\begin{displaymath}
0\leq V_t\leq C\int_0^t V_s ds, 
\end{displaymath}
which implies that $V_t=0$, $t\in[0,T]$ by the Gronwall inequality. Therefore, $\{X^n\}_{n\in\mathbb{N}}$ is a Cauchy sequence under the norm $(\hat{\mathbb{E}}[\sup_{s\in[0,T]}|\cdot|^p])^{1/p}$. Let $X$ be the limit of $\{X^n\}_{n\in\mathbb{N}}$. Clearly, $X\in S_G^p(0,T)$. Set 
\begin{displaymath}
{S}_t:=x_0+\int_0^t f(s,X_s)ds+\int_0^t h(s,X_s)d\langle B\rangle_s+\int_0^t g(s,X_s)dB_s.
\end{displaymath}
Let $A$ be the second component of the solution to the reflected problem with nonlinear constraints $l,r$ for $S$. 
Similar with \eqref{e5.7}, we have
\begin{displaymath}
\hat{\mathbb{E}}\left[\sup_{t\in[0,T]}|A_t-A^n_t|^p\right]\leq C\hat{\mathbb{E}}\left[\sup_{s\in[0,T]}|S_t-S^{n-1}_t|^p\right]\leq C\hat{\mathbb{E}}\left[\sup_{s\in[0,T]}|X_t-X^{n-1}_t|^p\right].
\end{displaymath}
Consequently, we have
\begin{displaymath}
\lim_{n\rightarrow\infty}\hat{\mathbb{E}}\left[\sup_{t\in[0,T]}|A_t-A^n_t|^p\right]=0,
\end{displaymath}
which implies that $A\in S_G^p(0,T)\cap M_{BV}(0,T)$. Then, $(X,A)$ is the solution of doubly reflected $G$-SDE with parameters $(x_0,f,h,g,l,r)$.
\end{proof}

\begin{remark}\label{remarkbackward}
 The doubly reflected backward SDEs driven by $G$-Brownian motion has been studied in \cite{LN,LiS}. The formulation of the backward problem is quite different from the one for the forward problem. First, for the backward problem, we can only deal with the case of linear constraints. That is, the constraints take the following form
 \begin{align*}
 l(t,x)=x-\beta_t, \ r(t,x)=x-\alpha_t,
 \end{align*}
where $\alpha,\beta\in S_G^p(0,T)$ represents the lower and the upper obstacle, respectively. Besides, the upper obstacle needs to almost be a  ``generalized $G$-It\^{o} process". Second,  due to the appearance of nonincreasing $G$-martingale in the backward problem, the bounded variation process in the solution cannot be separated completely and the solution is defined through a so-called ``approximate Skorokhod condition". However, for the forward reflected $G$-SDEs, the existence is proved based on a pathwise construction and a Picard iteration method. Therefore, we could decompose the bounded variation process into the difference between two increasing processes satisfying the Skorokhod conditions.
\end{remark}

\subsection{Comparison theorem}
In this subsection, we establish a comparison theorem for reflected $G$-SDEs with nonlinear constraints. First, recalling the non-reflected case, if the coefficients of the $dt$-term and $d\langle B\rangle$-term (i.e., the coefficients of the ``drift term") are larger, then the solution to the $G$-SDE is larger (see Theorem 5.5 in \cite{LW}). It is natural to conjecture that this fact still holds for the reflected $G$-SDEs. For the effect of the constraints, consider a special case that the solution $X$ is required to lie between two prescribed processes $\alpha,\beta$, called the lower obstacle and the upper obstacle, respectively. In this case, $l,r$ are given as in Remark \ref{remarkbackward}.  A natural observation is that when the obstacles are larger, the corresponding solution is larger, which lead to the conjecture that the smaller the constraints, the larger the solution. Combining the effect of both the coefficients and the constraints, we have the following result.

\begin{theorem}\label{thm5.9}
Given two reflected $G$-BSDEs with parameters $(x^i,f^i,h^i,g^i,l^i,r^i)$, $i=1,2$ satisfying Assumption \ref{ass1} and Assumption \ref{ass2}, we additionally suppose in the following:
\begin{itemize}
\item[(1)] $x^1\leq x^2$, $g^1=g^2=g$;
\item[(2)] for each $x\in\mathbb{R}$, $f^1(t,x)\leq f^2(t,x)$, $h^1(t,x)\leq h^2(t,x)$, $l^2(t,x)\leq l^1(t,x)$, $r^2(t,x)\leq r^1(t,x)$, $t\in[0,T]$, q.s.;
\item[(3)] there exists a constant $c>0$, such that for $i=1,2$, $t\in[0,T]$ and $x,y\in\mathbb{R}$, the following inequality holds q.s.,
\begin{align}\label{lowerlip}
c|x-y|\leq |r^i(t,\omega,x)-r^i(t,\omega,y)|\vee| l^i(t,\omega,x)-l^i(t,\omega,y)|;
\end{align}
\item[(4)] for each $n\in\mathbb{N}$ and $i=1,2$, there exists a pair $(l^{i,n},r^{i,n})$ satisfying Assumption \ref{ass1} and \eqref{lowerlip}, such that $(l^{i,n,})^{-1}(\cdot,0)$ and $(r^{i,n})^{-1}(\cdot,0)$ are bounded, $l^{2,n}(t,x)\leq l^{1,n}(t,x)$,  $r^{2,n}(t,x)\leq r^{1,n}(t,x)$, for any $n\in\mathbb{N}$, $t\in[0,T]$, q.s., and 
\begin{align*}
\lim_{n\rightarrow \infty}\hat{\mathbb{E}}\left[\sup_{t\in[0,T]}|(l^i)^{-1}(t,0)-(l^{i,n})^{-1}(t,0)|^p\right]=\lim_{n\rightarrow \infty}\hat{\mathbb{E}}\left[\sup_{t\in[0,T]}|(r^i)^{-1}(t,x)-(r^{i,n})^{-1}(t,0)|^p\right]=0.
\end{align*}
\end{itemize}
Let $(X^i,A^i)$ be the solution to the reflected $G$-SDE with parameters $(x^i,f^i,h^i,g^i,l^i,r^i)$, $i=1,2$. Then, we have
\begin{displaymath}
X_t^1\leq X_t^2, \ t\in[0,T], \textrm{ q.s.}
\end{displaymath}
\end{theorem}

\begin{proof}
\textbf{Step 1.} We first prove the case when $f^i,h^i,g$, $r^{i,-1}(\cdot,0)$ and $l^{i,-1}(\cdot,0)$ are bounded, $i=1,2$.

By Proposition \ref{prop5.1} and applying the boundedness of $f^i,h^i,g$, $r^{i,-1}(\cdot,0)$ and $l^{i,-1}(\cdot,0)$, we have $\hat{\mathbb{E}}[\sup_{t\in[0,T]}|A_t^i|^p]<\infty$ for any $p\geq 2$. 
Since $A^i\in S_G^p(0,T)\subset S_G^2(0,T)$, by Proposition \ref{the3.7}, for each $t\in[0,T]$, we have $\lim_{s\rightarrow t}\hat{\mathbb{E}}[|A^i_t-A^i_s|^2]=0$, $i=1,2$. Set $\hat{X}_t=X^1_t-X_t^2$. Then, we may apply the extended $G$-It\^{o}'s formula (see Theorem \ref{Gito}) to $(\hat{X}_t^+)^3$ to obtain that
\begin{equation}\begin{split}\label{e5.9}
(\hat{X}_t^+)^3=&3\int_0^t (\hat{X}_s^+)^2(f^1(s,X^1_s)-f^2(s,X^2_s))ds+3\int_0^t (\hat{X}_s^+)^2(h^1(s,X^1_s)-h^2(s,X^2_s))d\langle B\rangle_s\\
&+3\int_0^t (\hat{X}_s^+)^2(g(s,X^1_s)-g(s,X^2_s))dB_s+3\int_0^t (\hat{X}_s^+)^2d(A^1_s-A^2_s)\\
&+3\int_0^t \hat{X}_s^+(g(s,X^1_s)-g(s,X^2_s))^2d\langle B\rangle_s\\
\leq &3(L+\bar{\sigma}^2 L+\bar{\sigma}^2 L^2)\int_0^t (\hat{X}_s^+)^3ds+3\int_0^t (\hat{X}_s^+)^2d(A^1_s-A^2_s)\\
&+3\int_0^t (\hat{X}_s^+)^2(g(s,X^1_s)-g(s,X^2_s))dB_s.
\end{split}\end{equation}
Let $(A^{i,l},A^{i,r})$ be the constraining processes associated with reflected $G$-SDE $(x^i,f^i,h^i,g^i,l^i,r^i)$, $i=1,2$. It is easy to check that on the set $\{X_t^1>X^2_t\}$,
\begin{align*}
&X^1_t-X^2_t\leq \frac{1}{c}(r^2(t,X^1_t)-r^2(t,X^2_t))\leq \frac{1}{c}r^2(t,X^1_t)\leq \frac{1}{c} r^1(t,X^1_t),\\
&X^1_t-X^2_t\leq \frac{1}{c}(l^1(t,X^1_t)-l^1(t,X^2_t))\leq -\frac{1}{c}l^1(t,X^2_t)\leq -\frac{1}{c} l^2(t,X^2_t).
\end{align*}
It follows that 
\begin{align*}
\int_0^t (\hat{X}_s^+)^2d(A^1_s-A^2_s)=&\int_0^t (\hat{X}_s^+)^2dA^{1,r}_s-\int_0^t (\hat{X}_s^+)^2dA^{1,l}_s\\
&-\int_0^t (\hat{X}_s^+)^2dA^{2,r}_s+\int_0^t (\hat{X}_s^+)^2dA^{2,l}_s\\
\leq &\int_0^t \frac{1}{c^2}(r^1(s,X_s^1))^2dA^{1,r}_s-\int_0^t (\hat{X}_s^+)^2dA^{1,l}_s\\
&-\int_0^t (\hat{X}_s^+)^2dA^{2,r}_s+\int_0^t \frac{1}{c^2}(l^2(s,X^2_s))^2dA^{2,l}_s\\
\leq &-\int_0^t (\hat{X}_s^+)^2dA^{1,l}_s-\int_0^t (\hat{X}_s^+)^2dA^{2,r}_s\leq 0.
\end{align*}
Putting the above result into \eqref{e5.9} and then taking expectations on both sides yield that 
\begin{align*}
\hat{\mathbb{E}}[(\hat{X}_t^+)^3]\leq C\hat{\mathbb{E}}\left[\int_0^t (\hat{X}_s^+)^3 ds\right]\leq C\int_0^t \hat{\mathbb{E}}[(\hat{X}_s^+)^3]ds.
\end{align*}
Using the Gronwall inequality, we obtain the desired result.

\textbf{Step 2.} Now, we are in a position to prove the comparison property for the general case. For each fixed $n\in\mathbb{N}$, we define
\begin{align*}
	m^n(t,x):=((-n)\vee l(t,x))\wedge n, \ m=f^i,h^i,g, \ (t,x)\in[0,T]\times\mathbb{R},\  i=1,2.
\end{align*}
Let $((X^i)^n,(A^i)^n)$ be the solution of reflected $G$-SDE with initial value $x^i$, coefficients $(f^i)^n$, $(h^i)^n$, $g^n$ and nonlinear constraints $l^{i,n},r^{i,n}$, $i=1,2$. The result in the first step indicates that for each $n\in\mathbb{N}$,
\begin{equation}\label{e5.11}
	(X^1)^n_t\leq (X^2)^n_t, \ t\in[0,T], \textrm{ q.s.}
\end{equation}
Recalling Theorem \ref{prop5.2}, we have
\begin{align*}
	\hat{\mathbb{E}}\left[\sup_{t\in[0,T]}|(\hat{X}^i)^n_t|^p\right]\leq C&\Bigg\{\hat{\mathbb{E}}\left[\int_0^T|(\hat{f}^i)_t^n|^pdt\right]+\hat{\mathbb{E}}\left[\left(\int_0^T |\hat{g}^n_t|^2dt\right)^{p/2}\right]\\
	&+\hat{\mathbb{E}}\left[\int_0^T|(\hat{h}^i)^n_t|^p dt\right]+\hat{\mathbb{E}}[|(\bar{l}^i)^n_T|^p]+\hat{\mathbb{E}}[|(\bar{r}^i)^n_T|^p]\Bigg\},
\end{align*}
where for $i=1,2$,
\begin{align*}
	&(\hat{X}^i)^n_t=(X^i)^n_t-X^i_t, \ \hat{g}^n_t=g^n(t,X^i_t)-g(t,X^i_t),\\
	&(\hat{m}^i)^n_t=(m^i)^n(t,X^i_t)-m^i(t,X^i_t), \ m=f,h,\\
	&(\bar{l}^i)^n_T=\sup_{t\in[0,T]}|(l^{i,n})^{-1}(t,0)-(l^i)^{-1}(t,0)|,\\
	&(\bar{r}^i)^n_T=\sup_{t\in[0,T]}|(r^{i,n})^{-1}(t,0)-(r^i)^{-1}(t,0)|.
\end{align*}
It is easy to check that 
\begin{align*}
	\hat{\mathbb{E}}\left[\int_0^T|(\hat{f}^i)_t^n|^pdt\right]\leq &\hat{\mathbb{E}}\left[\int_0^T |f^i(t,X^i_t)|^p I_{\{ |f^i(t,X_t^i)|>n \}}dt\right]\\
	\leq &\hat{\mathbb{E}}\left[\int_0^T( |f^i(t,0)|+L|X^i_t|)^p I_{\{ |f^i(t,0)|+L|X^i_t|>n \}}dt\right]\\
	\leq & C\left\{\hat{\mathbb{E}}\left[\int_0^T |f^i(t,0)|^p I_{\{ |f^i(t,0)|>\frac{n}{2} \}}dt\right] +\hat{\mathbb{E}}\left[\int_0^T |X^i_t|^p I_{\{ L|X^i_t|>\frac{n}{2} \}}dt\right]\right \}.
\end{align*}
Noting that $f^i,X^i\in M_G^p(0,T)$, by Theorem \ref{thm4.7}, we have 
\begin{displaymath}
	\lim_{n\rightarrow \infty}	\hat{\mathbb{E}}\left[\int_0^T|(\hat{f}^i)_t^n|^pdt\right]=0.
\end{displaymath}
Similarly, we have 
\begin{displaymath}
		\lim_{n\rightarrow \infty}	\hat{\mathbb{E}}\left[\left(\int_0^T |\hat{g}^n_t|^2dt\right)^{p/2}\right]=	\lim_{n\rightarrow \infty}	\hat{\mathbb{E}}\left[\int_0^T|(\hat{h}^i)_t^n|^pdt\right]=0.
\end{displaymath}
All the above analysis implies that 
\begin{align*}
	\lim_{n\rightarrow\infty}\hat{\mathbb{E}}\left[\sup_{t\in[0,T]}|(X^i)^n_t-X^i_t|^p\right]=0.
\end{align*}
Letting $n$ approach infinity in \eqref{e5.11}, we finally obtain the desired result.
\end{proof}

For $i=1,2$, let $l^i,r^i$ be given as follows
\begin{align}\label{liri}
l^i(t,x)=x-\beta^i_t, \ r^i(t,x)=x-\alpha^i_t,
\end{align}
where $\alpha^i,\beta^i\in S_G^p(0,T)$ satisfy $\inf_{(t,\omega)\in [0,T]\times \Omega_T}(\beta^i_t(\omega)-\alpha^i_t(\omega))>0$ and $\alpha^1_t\leq \alpha^2_t$, $\beta^1_t\leq \beta^2_t$, $t\in[0,T]$. Then, we claim that  $(l^i,r^i)$ satisfies conditions (3), (4) in Theorem \ref{thm5.9}. In fact, condition (3) is easy to verify. It remains to construct sequences $\{l^{i,n}\}_{n\in\mathbb{N}}$ and $\{r^{i,n}\}_{n\in\mathbb{N}}$ satisfying condition (4). For any $n\in\mathbb{N}$, set 
\begin{align*}
l^{i,n}(t,x)=x-\frac{1}{n}-\beta^{i,n}_t, \ r^{i,n}(t,x)=x+\frac{1}{n}-\alpha^{i,n}_t,
\end{align*}
where $\beta^{i,n}_t=(\beta^i_t\wedge n)\vee(-n)$ and $\alpha^{i,n}_t=(\alpha^i_t\wedge n)\vee(-n)$. For simplicity, we omit the superscript $i$. Simple calculation yields that 
\begin{align*}
(l^n)^{-1}(t,0)=\frac{1}{n}+\beta^n_t, \ (r^n)^{-1}(t,0)=-\frac{1}{n}+\alpha^n_t,
\end{align*}
which are bounded and satisfy 
\begin{align*}
\inf_{(t,\omega)\in [0,T]\times \Omega_T}((l^n)^{-1}(t,\omega,0)-(r^n)^{-1}(t,\omega,0))>0.
\end{align*}
Therefore, $(l^n,r^n)$ satisfy Assumption \ref{ass1} and Equation \eqref{lowerlip}. Note that 
\begin{align*}
\hat{\mathbb{E}}\left[\sup_{t\in[0,T]}|l^{-1}(t,0)-(l^{n})^{-1}(t,0)|^p\right]&\leq C\left\{\frac{1}{n^p}+\hat{\mathbb{E}}\left[\sup_{t\in[0,T]}|\beta^n_t-\beta_t|^p\right]\right\}\\
&\leq C\left\{\frac{1}{n^p}+\hat{\mathbb{E}}\left[\sup_{t\in[0,T]}|\beta_t|^p I_{\{\sup_{t\in[0,T]}|\beta_t|>n\}}\right]\right\}.
\end{align*}
Since we have $\sup_{t\in[0,T]}|\beta_t|\in L_G^p(\Omega_T)$, applying Theorem \ref{LGp} yields that
\begin{align*}
\lim_{n\rightarrow \infty}\hat{\mathbb{E}}\left[\sup_{t\in[0,T]}|l^{-1}(t,0)-(l^{n})^{-1}(t,0)|^p\right]=0.
\end{align*}
Similarly, we have
\begin{align*}
\lim_{n\rightarrow \infty}\hat{\mathbb{E}}\left[\sup_{t\in[0,T]}|r^{-1}(t,0)-(r^{n})^{-1}(t,0)|^p\right]=0.
\end{align*}
Therefore, the pair $(l^i,r^i)$ satisfies condition (3) in Theorem \ref{thm5.9}. Combining the above analysis and Theorem \ref{thm5.9}, we obtain the comparison property, which extends Theorem 5.9 in \cite{Lin} to the case of double obstacles.

\begin{corollary}
Given two reflected $G$-BSDEs with parameters $(x^i,f^i,h^i,g^i,\alpha^i,\beta^i)$, $i=1,2$ satisfying  Assumption \ref{ass2}, we additionally suppose in the following:
\begin{itemize}
\item[(1)] $x^1\leq x^2$, $g^1=g^2=g$;
\item[(2)] for $i=1,2$, $\alpha^i,\beta^i\in S_G^p(0,T)$ satisfy 
$$\inf_{(t,\omega)\in [0,T]\times \Omega_T}(\beta^i_t(\omega)-\alpha^i_t(\omega))>0;$$
\item[(3)] for each $x\in\mathbb{R}$, $f^1(t,x)\leq f^2(t,x)$, $h^1(t,x)\leq h^2(t,x)$, $\alpha^1_t\leq \alpha^2_t$, $\beta^1_t\leq \beta^2_t$, $t\in[0,T]$, q.s.;
\end{itemize}
Let $(X^i,A^i)$ be the solution to the reflected $G$-SDE with parameters $(x^i,f^i,h^i,g^i,\alpha^i,\beta^i)$, $i=1,2$. Then, we have
\begin{displaymath}
X_t^1\leq X_t^2, \ t\in[0,T], \textrm{ q.s.}
\end{displaymath}
\end{corollary}

In the following, we consider the monotonicity of the individual constraining processes for reflected $G$-SDEs with nonlinear constraints.
\begin{proposition}
Given two reflected $G$-BSDEs with parameters $(x^i,f^i,h^i,g,l,r)$, $i=1,2$ satisfying the assumptions in Theorem \ref{thm5.9}, we additionally suppose the following condition.
\begin{itemize}
\item  For each $t\in[0,T]$, $g$ does not depends on $x$ and either $f^1(t,\cdot),g^1(t,\cdot)$ or $f^2(t,\cdot),g^2(t,\cdot)$ are nondecreasing.
\end{itemize}
Let $(X^i,A^i)$ be the solution to reflected $G$-SDE with parameters $(x^i,f^i,h^i,g,l,r)$ and $(A^{i,l},A^{i.r})$ be the associated constraining processes, $i=1,2$. Then, we have
\begin{itemize}
\item[(1)] $A^{1,l}_t\leq A^{2,l}_t\leq A^{1,l}_t+\hat{X}_t+(x^2-x^1)$;
\item[(2)] $A^{2,r}_t\leq A^{1,r}_t\leq A^{2,r}_t+\hat{X}_t+(x^2-x^1)$,
\end{itemize}
where $\hat{X}_t=\int_0^t f^2(s,X^2_s)-f^1(s,X^1_s)ds+\int_0^t h^2(s,X^2_s)-h^1(s,X^1_s)d\langle B\rangle_s$.
\end{proposition}

\begin{proof}
For any $t\in[0,T]$, $i=1,2$, set 
\begin{align*}
\tilde{S}^i_t=\int_0^t f^i(s,X^i_s)ds+\int_0^t h^i(s,X^i_s)d\langle B\rangle_s+\int_0^t g(s)dB_s.
\end{align*} 
By Theorem \ref{thm5.9}, we have $X^1_t\leq X^2_t$, $t\in[0,T]$. Without loss of generality, suppose that $f^2(t,\cdot),g^2(t,\cdot)$ are nondecreasing. Since
\begin{align*}
\hat{X}_t=&\int_0^t f^2(s,X^2_s)-f^2(s,X^1_s)ds+\int_0^t h^2(s,X^2_s)-h^2(s,X^1_s)d\langle B\rangle_s\\
&+\int_0^t f^2(s,X^1_s)-f^1(s,X^1_s)ds+\int_0^t h^2(s,X^1_s)-h^1(s,X^1_s)d\langle B\rangle_s.
\end{align*} 
it follows that $\{\hat{X}_t\}_{t\in[0,T]}$ is a nondecreasing process. Note that $(X^i,A^i)$ can be viewed as the solution of the reflected problem with constraints $l,r$ for $x^i+\tilde{S}^i$, $i=1,2$ and $\tilde{S}^2_t-\tilde{S}^1_t=\hat{X}_t$. By Proposition \ref{prop3.7}, we obtain the desired result.
\end{proof}

\appendix
\renewcommand\thesection{Appendix}
\section{ }
\renewcommand\thesection{A}
\setcounter{equation}{0}
\renewcommand{\theequation}{A\arabic{equation}}

It is clear that $\int_0^T X_t dK_t$ given in Definition \ref{deflin} belongs to $L^0(\Omega_T)$. In the following, we will show under which condition, the random variable $\int_0^T X_t dK_t$ belongs to the $L_G^1(\Omega_T)$. Since the element $\xi$ in $L_G^1(\Omega_T)$ has a quasi-continuous version (see Theorem \ref{LGp}), except the integrability condition, we still need some continuity property of $X,K$ to ensure that  $\int_0^T X_t dK_t\in L_G^1(\Omega_T)$.

\begin{proposition}
For any $p>1$, suppose that $X\in S_G^p(0,T)$, $K\in M_I(0,T)$ and for any $t\in[0,T]$ $K_t\in L_G^q(\Omega_t)$, where $1/p+1/q=1$. Then, we have $\int_0^T X_t dK_t\in L_G^1(\Omega_T)$.
\end{proposition}

\begin{proof}
Let $\{\pi^n_{[0,T]}\}_{n\in \mathbb{N}}$ be a sequence of refining partitions of $[0,T]$ (i.e., $\pi^n_{[0,T]}\subset \pi^{n+1}_{[0,T]}$, for any $n\in\mathbb{N}$) with $\pi^n_{[0,T]}=\{t^n_0,t^n_1,\dots, t^n_n\}$ such that $\lim_{n\rightarrow \infty}\mu(\pi^n_{[0,T]})=0$, where 
\begin{displaymath}
\mu(\pi^n_{[0,T]}):=\max_{k=0,1,\dots,n-1}|t_{k+1}^n-t^n_k|.
\end{displaymath}
Set 
\begin{displaymath}
\mathcal{V}^n_{[0,T]}(X,K)(\omega):=\sum_{k=0}^{n-1} X_{t^n_k}(\omega)(K_{t^n_{k+1}}(\omega)-K_{t^n_k}(\omega)).
\end{displaymath}
For each $n\in\mathbb{N}$, it is easy to check that  $\mathcal{V}^n_{[0,T]}(X,K)\in L_G^1(\Omega_T)$. Simple calculation yields that 
\begin{align*}
\left|\mathcal{V}^n_{[0,T]}(X,K)-\int_0^T X_t dK_t\right|\leq & \int_0^T\left|\sum_{k=0}^{n-1}X_{t^n_k} I_{[t_k^n,t^n_{k+1})}-X_t\right|dK_t\\
\leq &K_T \sup_{k=0,1,\dots,n-1}\sup_{t\in[t^n_k,t^n_{k+1})}|X_t-X_{t^n_k}|.
\end{align*}
By Proposition \ref{the3.7} and the H\"{o}lder inequality, we obtain that 
\begin{align*}
&\lim_{n\rightarrow \infty}\hat{\mathbb{E}}\left[\left|\mathcal{V}^n_{[0,T]}(X,K)-\int_0^T X_t dK_t\right|\right]\\
\leq& 
\lim_{n\rightarrow \infty}\left(\hat{\mathbb{E}}\left[\sup_{k=0,1,\dots,n-1}\sup_{t\in[t^n_k,t^n_{k+1})}|X_t-X_{t^n_k}|^p\right]\right)^{1/p}\left(\hat{\mathbb{E}}[|K_T|^q]\right)^{1/q}=0,
\end{align*}
which implies the desired result.
\end{proof}

\begin{remark}
Compared with Proposition 3.13 in \cite{Lin}, we do not need to assume that $X$ is a $G$-It\^{o} process taking the following form
\begin{displaymath}
X_t=x+\int_0^t f_s ds+\int_0^t h_s \langle B\rangle_s+\int_0^t g_s dB_s,
\end{displaymath}
where $f,h\in M_G^p(0,T)$ and $g\in H_G^p(0,T)$ with some $p>2$.
\end{remark}



\begin{thebibliography}{00}



\bibitem{BKR} Burdzy K, Kang W, Ramanan K. The Skorokhod problem in a time-dependent interval. Stochastic Processes and their Applications, 2009, 119: 428-452.

\bibitem{BN} Burdzy K, Nualart D. Brownian motion reflected on Brownian motion. Probab. Theor. Rel. Fields, 2002, 122: 471-493.

\bibitem{CE} Chaleyat-Maurel M, El Karoui N.  Un probl\`{e}me de r\'{e}flexion et ses applications au temps local et aux \'{e}quations diff\'{e}rentielles stochastiques sur $\mathbb{R}$, cas continu. Expos\'{e}s du S\'{e}minaire J. Az\'{e}ma-M. Yor. Held at the Universit\'{e} Pierre et Marie Curie, Paris, 1976-1977, pp. 117–144. Ast\'{e}risque, 52, 53, Soci\'{e}t\'{e} Math\'{e}matique de France, Paris, 1978.

\bibitem {DHP11} Denis L, Hu M, Peng S. Function spaces and capacity related to a sublinear expectation: application to $G$-Brownian motion paths. Potential Anal., 2011, 34: 139-161.



\bibitem{EK} El Karoui N, Karatzas I. A new approach to the Skorohod problem, and its applications. Stochastics Stochastics Rep., 1991, 34: 57-82.

\bibitem{G} Gao F. Pathwise properties and homeomorphic flows for stochastic differential equations driven by $G$-Brownian motion. Stochastic Process. Appl., 2009, 119: 3356–3382.















\bibitem{HH} Henderson V, Hobson D. Local time, coupling and the passport option. Finance Stoch., 2000, 4: 69-80.



\bibitem{HWZ} Hu M, Wang F, Zheng G. Quasi-continuous random variables and processes under the $G$-expectation framework. Stochastic Process. Appl., 2016, 126(8): 2367-2387.

\bibitem{JO} Jarni I, Ouknine Y. On reflection with two-sided jumps. Journal of Theoretical Probability, 2021, 34: 1811-1830.

\bibitem{KLRS} Kruk L, Lehoczky J, Ramanan K, Shreve S. An explicit formula for the Skorokhod map on $[0,a]$. Annals of Probability, 2007, 35(5): 1740-1768.

   \bibitem{Li} Li H. The Skorokhod problem with two nonlinear constraints. Probability and Mathematical Statistics, 2023, 43(2): 207-239.

   \bibitem{LN23} Li H, Ning N. Stochastic differential equations driven by $G$-Brownian motion with mean reflections, 2023, arXiv: 2306.08931v3.
   
   \bibitem{LN} Li H, Ning N. Doubly reflected backward SDEs driven by G-Brownian motions and fully nonlinear PDEs with double obstacles. Stochastics and Partial Differential Equations: Analysis and Computations, 2025, 13: 1279-1318.

   \bibitem{LPS} Li H, Peng S, Song Y. Supermartingale decomposition theorem under $G$-expectation. Electron. J. Probab., 2018, 23: 1-20.
  
   
   \bibitem{LiS} Li H, Song Y. Backward stochastic differential equations driven by $G$-Brownian motion with double reflections. Journal of Theoretical Probability, 2021, 34: 2285-2314.
   
   
   \bibitem{lp} Li X, Peng S. Stopping times and related It\^{o}'s calculus with $G$-Brownian motion. Stochastic Process. Appl., 2011, 121: 1492-1508.
   


\bibitem{Lin} Lin Y. Stochastic differential equations driven by $G$-Brownian motion with reflecting boundary conditions. Electron. J. Probab., 2013, 18: 1-24.

\bibitem{Lin2} Lin Y. \'{E}quations diff\'{e}rentielles stochastiques sous les esp\'{e}rances math\'{e}matiques non-lin\'{e}aires et applications. PhD thesis, 2013, Universit\'{e} de Rennes. 

\bibitem{LSH} Lin Y, Soumana Hima A.  Reflected stochastic differential equations driven by $G$-Brownian motion in non-convex domains. Stochastics and Dynamics, 2019, 19(3): 1950025.

\bibitem{LS} Lions P L,  Sznitman A L. Stochastic differential equations with reflecting boundary conditions. Comm. Pure Appl. Math., 1984, 37: 511–537.

\bibitem{Luo} Luo P. Reflected stochastic differential equations driven by $G$-Brownian motion with nonlinear resistance. Front. Math. China, 2016, 11(1): 123–140.

\bibitem{LW} Luo P, Wang F. Stochastic differential equations driven by $G$-Brownian motion and ordinary differential equations. Stochastic Process. Appl., 2014, 124(11): 3869-3885.




\bibitem{MM} Mandelbaum A, Massey W. Strong approximations for time-dependent queues. Math. Oper. Res., 1995, 20: 33-63.











\bibitem{P19} Peng S. Nonlinear Expectations and Stochastic Calculus Under Uncertainty: With Robust CLT and G-Brownian Motion. Probability Theory and Stochastic Modelling 95, Springer, 2019.

\bibitem{Skorokhod1} Skorokhod A.V. Stochastic equations for diffusions in a bounded region. Theory Probab. Appl., 1961, 6: 264-274.


\bibitem{S1} Slaby M. Explicit representation of the Skorokhod map with time dependent boundaries. Probability and Mathematical Statistics, 2010, 30: 29-60.
 
\bibitem{S2} Slaby M. An explicit representation of the extended Skorokhod map with two time-dependent boundaries. Journal of Probability and Statistics, 2010, 2010: 1-18.


\bibitem{Sl1} Slomi\'{n}ski L. On existence, uniqueness and stability of solutions of multidimensional SDE’s with reflecting conditions. Ann. lIHP Probab. Stat., 1993, 29: 163-198. 

\bibitem{Sl2} Slomi\'{n}ski L. On approximation of solutions of multidimensional SDEs with reflecting boundary conditions. Stoch. Process. Appl., 1994, 50: 197-219.

\bibitem{SW1} Slomi\'{n}ski L, Wojciechowski T. Stochastic differential equations with jump reflection at time-dependent barriers. Stoch. Process. Appl., 2010, 120: 1701-1721.

\bibitem{SW2} Slomi\'{n}ski L, Wojciechowski T.  Stochastic differential equations with time-dependent reflecting barriers. Stochastics, 2013, 85(1): 27-47.

\bibitem{SW} Soucaliuc F, Werner W. A note on reflecting Brownian motions. Electon. Commum. Probab., 2002, 7: 117-122.

\bibitem{SHD} Soumana Hima A, Dakaou I. Large deviation principle for reflected stochastic differential equations driven by $G$-Brownian motion in non-convex domains. Statistics and Probability Letters, 2023, 193: 109707.

\bibitem{Tanaka} Tanaka, H. Stochastic differential equations with reflecting boundary condition in convex regions. Hiroshima Math. J., 1979, 9: 163-177.

\bibitem{W} Whitt W. An Introduction to Stochastic Process Limits and Their Application to Queues. Springer, New York, 2002.


\end{thebibliography}




\end{document}